\documentclass[10pt,a4paper,final]{article}

\usepackage[hmargin={27mm,27mm},vmargin={30mm,35mm}]{geometry}
\usepackage{amsfonts}
\usepackage{amssymb}
\usepackage{amsthm}
\usepackage{bm}
\usepackage{mathtools}
\usepackage{latexsym}
\usepackage{graphicx}
\usepackage{booktabs}
\usepackage{pgfplots,pgfplotstable}
\pgfplotsset{compat=1.16}
\usepackage{xcolor}
\usepackage{setspace}
\usepackage{url}
\usepackage{float}
\usepackage{subcaption}

\usepackage[colorlinks,allcolors={blue},linkcolor={black}]{hyperref}

\numberwithin{equation}{section}
\usepackage[color,notref,notcite]{showkeys}

\usepackage[normalem]{ulem}
\newcounter{corr}
\definecolor{violet}{rgb}{0.580,0.,0.827}
\newcommand{\corr}[4][]{\typeout{Warning : a correction remains in page \thepage}
 \stepcounter{corr}
	      {\color{blue}\ifmmode\text{\,\sout{\ensuremath{#2}}\,}\else\sout{#2}\fi}
        {\color{green!50!black}#3}
        {\color{violet}#4}
}

\usepackage{authblk}
\newcommand{\email}[1]{\href{mailto:#1}{#1}}

\newtheorem{theorem}{Theorem}

\newtheorem{proposition}[theorem]{Proposition}
\newtheorem{lemma}[theorem]{Lemma}
\newtheorem{corollary}[theorem]{Corollary}

\newtheorem{assumption}{Assumption}

\newtheorem{remark}{Remark}

\newcommand{\bmrm}[1]{{\bm{{\rm #1}}}}

\newcommand{\mat}[1]{\bmrm{#1}}

\newcommand{\bmn}{{\bm{n}}}

\newcommand{\bmx}{{\bm{x}}}

\newcommand{\calE}{\mathcal{E}}
\newcommand{\calF}{\mathcal{F}}

\newcommand{\calH}{\mathcal{H}}

\newcommand{\calM}{\mathcal{M}}

\newcommand{\calT}{\mathcal{T}}

\newcommand{\bbN}{\mathbb{N}}

\newcommand{\bbP}{\mathbb{P}}

\newcommand{\bbR}{\mathbb{R}}

\newcommand{\rma}{{\rm{a}}}
\newcommand{\rmb}{{\rm{b}}}

\newcommand{\rmp}{{\rm{p}}}

\newcommand{\rms}{{\rm{s}}}

\newcommand{\eq}{ ={}& }
\newcommand{\lea}{ \le{}& }

\newcommand{\les}{ \lesssim{}& }

\newcommand{\nn}{\nonumber}
\newcommand{\nl}{\nn\\}

\newcommand{\defeq}{\vcentcolon=}

\newcommand{\ul}[1]{\underline{#1}}
\newcommand{\ol}[1]{\overline{#1}}

\newcommand{\CARD}[1]{\mbox{\textrm{Card}}\big(#1\big)}

\newcommand{\bdry}{\partial}

\newcommand{\Mh}[1][h]{\calM_{#1}}
\newcommand{\Th}[1][h]{\calT_{#1}}
\newcommand{\Fh}[1][h]{\calF_{#1}}
\newcommand{\Fhb}{\Fh^{\rmb}}

\newcommand{\T}{{T}}
\newcommand{\F}{{F}}
\newcommand{\FT}{{\Fh[\T]}}

\newcommand{\hT}{h_\T}
\newcommand{\hF}{h_\F} 
\newcommand{\hX}{h_X}

\newcommand{\bdryT}{{\bdry \T}}

\newcommand{\matK}{\mat{K}}
\newcommand{\matKT}{\matK_\T}

\newcommand{\olK}{\ol{K}}
\newcommand{\ulK}{\ul{K}}
\newcommand{\olKT}{\olK_\T}
\newcommand{\ulKT}{\ulK_\T}

\newcommand{\alphaT}{\alpha_\T}

\newcommand{\nor}{\bmn}
\newcommand{\norT}{\nor_{\bdryT}}

\newcommand{\norTF}{\nor_{\T\F}}

\def\R{\bbR}

\def\N{\bbN}

\newcommand{\POLY}[1]{\bbP^{#1}}
\newcommand{\HS}[1]{H^{#1}}
\newcommand{\HONE}{\HS{1}}
\newcommand{\HONEzr}{\HONE_0}
\newcommand{\WSP}[1]{W^{#1}}
\newcommand{\LP}[1]{L^{#1}}
\newcommand{\LTWO}{\LP{2}}

\newcommand{\norm}[2][]{\|#2\|_{#1}}
\newcommand{\seminorm}[2][]{|#2|_{#1}}
\newcommand{\brac}[2][]{(#2)_{#1}}
\newcommand{\jump}[2][]{[\![#2]\!]_{#1}}

\newcommand{\SEMINORM}[2][]{\left|#2\right|_{#1}}
\newcommand{\BRAC}[2][]{\left(#2\right)_{#1}}

\newcommand{\piT}[1]{\pi_T^{#1}}
\newcommand{\piTzr}[1]{\piT{0, #1}}

\newcommand{\pih}[1]{\pi_h^{#1}}
\newcommand{\pihzr}[1]{\pih{0, #1}}

\newcommand{\piFT}[1]{\pi_\FT^{#1}}
\newcommand{\piFTzr}[1]{\piFT{0, #1}}

\newcommand{\piF}[1]{\pi_F^{#1}}
\newcommand{\piFzr}[1]{\piF{0, #1}}

\newcommand{\piX}[1]{\pi_X^{#1}}

\newcommand{\piKT}[1]{\pi_{\matK, T}^{#1}}
\newcommand{\piKTe}[1]{\piKT{1, #1}}

\newcommand{\piKh}[1]{\pi_{\matK, h}^{#1}}
\newcommand{\piKhe}[1]{\piKh{1, #1}}

\newcommand{\pKT}[1]{\rmp_{\matK,T}^{#1}}

\newcommand{\pKh}[1]{\rmp_{\matK,h}^{#1}}

\newcommand{\U}[2]{\ul{U}_{#1}^{#2}}

\newcommand{\UT}[1]{\U{T}{#1}}

\newcommand{\UTkl}{\UT{k,l}}

\newcommand{\Uh}[1]{\U{h}{#1}}

\newcommand{\Uhkl}{\Uh{k,l}}
\newcommand{\Uhklzr}{\U{h, 0}{k,l}}

\newcommand{\I}[2]{\ul{I}_{#1}^{#2}}

\newcommand{\IT}[1]{\I{T}{#1}}

\newcommand{\ITkl}{\IT{k,l}}

\newcommand{\Ih}[1]{\I{h}{#1}}

\newcommand{\Ihkl}{\Ih{k,l}}

\def\a{\rma}

\newcommand{\aKT}{\a_{\matK,T}}
\newcommand{\sKT}{\rms_{\matK,T}}
\newcommand{\aKh}{\a_{\matK,h}}
\newcommand{\sKh}{\rms_{\matK,h}}

\newcommand{\sKTgrad}{\sKT^{\nabla}}
\newcommand{\sKTbdry}{\sKT^{\partial}}
\newcommand{\sKTgradmin}{\sKT^{\nabla, {\rm min}}}
\newcommand{\sKTgradmax}{\sKT^{\nabla, {\rm max}}}
\newcommand{\sKTkminus}{\sKT^{(k-1)}}

\newcommand{\vT}{v_T}
\newcommand{\vh}{v_h}
\newcommand{\vF}{v_F}
\newcommand{\vFT}{v_{\FT}}

\newcommand{\uT}{u_T}

\newcommand{\uF}{u_F}

\newcommand{\ulvT}{\ul{v}_T}
\newcommand{\uluT}{\ul{u}_T}
\newcommand{\ulvh}{\ul{v}_h}
\newcommand{\uluh}{\ul{u}_h}

\newcommand{\deltaKT}[1]{\delta_{\matK,T}^{#1}}

\newcommand{\deltaKFT}[1]{\delta_{\matK,\FT}^{#1}}

\newcommand{\locerrorRHS}[1]{\alphaT\Big[\hT^{r+1}\seminorm[\matK, \HS{r+2}(T)]{#1}\Big]^2}
\newcommand{\errorRHS}[1]{\BRAC{\sum_{T\in\Th}\locerrorRHS{#1}}^\frac12}

\newcommand{\logLogSlopeTriangle}[5]
{
	\pgfplotsextra
	{
		\pgfkeysgetvalue{/pgfplots/xmin}{\xmin}
		\pgfkeysgetvalue{/pgfplots/xmax}{\xmax}
		\pgfkeysgetvalue{/pgfplots/ymin}{\ymin}
		\pgfkeysgetvalue{/pgfplots/ymax}{\ymax}
		
		\pgfmathsetmacro{\xArel}{#1}
		\pgfmathsetmacro{\yArel}{#3}
		\pgfmathsetmacro{\xBrel}{#1-#2}
		\pgfmathsetmacro{\yBrel}{\yArel}
		\pgfmathsetmacro{\xCrel}{\xArel}
		
		\pgfmathsetmacro{\lnxB}{\xmin*(1-(#1-#2))+\xmax*(#1-#2)} 
		\pgfmathsetmacro{\lnxA}{\xmin*(1-#1)+\xmax*#1} 
		\pgfmathsetmacro{\lnyA}{\ymin*(1-#3)+\ymax*#3} 
		\pgfmathsetmacro{\lnyC}{\lnyA+#4*(\lnxA-\lnxB)}
		\pgfmathsetmacro{\yCrel}{\lnyC-\ymin)/(\ymax-\ymin)}
		
		\coordinate (A) at (rel axis cs:\xArel,\yArel);
		\coordinate (B) at (rel axis cs:\xBrel,\yBrel);
		\coordinate (C) at (rel axis cs:\xCrel,\yCrel);
		
		\draw[#5]   (A)-- node[pos=0.5,anchor=north] {\scriptsize{1}}
		(B)-- 
		(C)-- node[pos=0.,anchor=west] {\scriptsize{#4}} 
		cycle;
	}
}

\begin{document}	
\title{Robust Hybrid High-Order method on polytopal meshes with small faces}
\author[1]{J\'er\^ome Droniou}
\author[2]{Liam Yemm}
\affil[1]{School of Mathematics, Monash University, Melbourne, Australia, \email{jerome.droniou@monash.edu}}
\affil[2]{School of Mathematics, Monash University, Melbourne, Australia, \email{liam.yemm@monash.edu}}

\maketitle

\begin{abstract}
  We design a Hybrid High-Order (HHO) scheme for the Poisson problem that is fully robust on polytopal meshes in the presence of small edges/faces. We state general assumptions on the stabilisation terms involved in the scheme, under which optimal error estimates (in discrete and continuous energy norms, as well as \(\LTWO\)-norm) are established with multiplicative constants that do not depend on the maximum number of faces in each element, or the relative size between an element and its faces. We illustrate the error estimates through numerical simulations in 2D and 3D on meshes designed by agglomeration techniques (such meshes naturally have elements with a very large numbers of faces, and very small faces).
\end{abstract}

	
\section{Introduction}\label{sec:introduction}

In this paper, we design a Hybrid High-Order (HHO) scheme for diffusion problems that is fully robust on polytopal meshes in the presence of small edges/faces. 

Hybrid-High Order schemes form a family of polytopal methods, that is, numerical methods for diffusion problems that can be applied to meshes made of generic polygonal (in 2D) or polyhedral (in 3D) elements.	Additionally, HHO methods are of arbitrary order, which is to say they can achieve any preset level of accuracy by a proper selection of the degrees of their polynomial unknowns. In recent years, there has been a growing interest in polytopal methods, both of low- and arbitrary-order. A non-exhaustive list includes Discontinuous Galerkin and Hybridizable Discontinuous Galerkin methods \cite{di-pietro.ern:2011:mathematical,cangiani.dong.ea:2017:discontinuous,cockburn.dong.ea:2009:hybridizable}, Multi-Point Flux Approximation Finite Volume methods \cite{aavatsmark.eigestad.ea:2008:compact}, Hybrid Mimetic Mixed methods \cite{droniou.eymard:2010:unified} (which include Mixed/Hybrid Mimetic Finite Differences \cite{beirao-da-veiga.lipnikov.ea:2014:mimetic}, the SUSHI scheme \cite{eymard.gallouet.ea:2010:discretization} and Mixed Finite Volumes \cite{droniou.eymard:2006:mixed}), Virtual Element method \cite{beirao-da-veiga.brezzi.ea:2013:basic,ahmad.alsaedi.ea:2013:equivalent,brezzi.falk.ea:2014:basic,cangiani.manzini.ea:2017:conforming}, Weak Galerkin methods \cite{mu.wang.ea:2015:weak}, and polytopal Finite Elements \cite{sukumar.tabarraei:2004:conforming}. We refer the reader to the introduction of \cite{di-pietro.droniou:2020:hybrid} for a thorough literature review of polytopal methods. Among those, HHO has specific features that set them apart: they are built on polynomial reconstructions that account for the local physics and enable robustness with respect to the model's parameters (such as dominating advection in diffusion--advection models \cite{di-pietro.droniou.ea:2014:discontinuous}); their design is dimension-independent; and they are amenable to local static condensation (which drastically reduces the number of globally coupled degrees of freedom). Moreover, they can be recast in finite volume form, using numerical fluxes that satisfy the foundational properties of finite volume methods: conservativity and local balance \cite{droniou:2014:finite}; such fluxes prove particularly useful for coupled flow problems \cite{anderson.droniou:2018:arbitrary,botti.di-pietro.ea:2019.hybrid}. The original framework of the HHO method can be traced back to \cite{di-pietro.ern.ea:2014:arbitrary} which formulates a hybrid method of arbitrary order for the Poisson problem compatible with general polytopal meshes. These principles were extended to a linear elasticity problem in \cite{di-pietro.ern:2015:hybrid} and referred to as a `Hybrid High-Order method'. A comprehensive overview of the method and its applications can be found in \cite{di-pietro.droniou:2020:hybrid}.

Analysis of polytopal methods is usually carried out under mesh regularity assumptions that require, as the mesh is refined, that the elements and faces do not stretch in one direction, and that the faces have a comparable size to their elements. Discontinuous Galerkin (DG) schemes have already been shown to be robust with respect to small/numerous faces (c.f. \cite[Section 4.3]{cangiani.dong.ea:2017:discontinuous}). More recently, DG schemes have been applied on near arbitrary meshes possessing, possibly, curved elements \cite{cangiani.dong:2019:version}. Some methods, such as Virtual Element methods (VEM), require specific design choices of the stabilisation terms to ensure that the error estimates are independent of the presence of small faces \cite{brenner.sung:2018:virtual,brenner.guan.ea:2017:some,beirao-da-veiga.lovadini:2017:stability}. We also note the recent work \cite{beirao-da-veiga.vacca:2020:sharper} in which an error estimate for VEM is established, that separate the different contributions of element and edge unknowns and shed some light on the role of the different polynomial degrees of these unknowns. None of these methods are however finite volume methods, for which faces play a particular role. In the context of HHO, an analysis on skewed meshes has been carried out in \cite{droniou:2020:interplay}, and identifies how the error estimate is impacted by the element distortion and local diffusion tensor. This analysis however does not lead to robust estimates in terms of small faces in distorted meshes, or meshes with small faces but otherwise regular (``round'') elements. We also note the recent work \cite{bertoluzza.manzini.ea:2021}, made public shortly after our work, on the design of stabilisations for non-conforming VEM (linked to HHO \cite[Section 5.5]{di-pietro.droniou:2020:hybrid}) that allow for error estimates robust with respect to small edges; these stabilisations are based on the fully discrete representation of ncVEM (closer to the HHO presentation than to standard VEM presentations), but are restricted to the 2D setting and require a more complex abstract construction.

Meshes with complex polytopal elements are often unavoidable in applications. Solution techniques (e.g. multi-grid algorithms) for the linear systems arising from the scheme may require to consider meshes comprising of agglomerated elements, with many faces that are much smaller than the elements themselves. Such grids naturally arise when meshing domains with complex geometries, such as in subsurface fluid flow \cite{yeh.yeh:2000:computational}. The existence of numerous thin layers within the subsoil, as well as faults, wells, and other complex geometry are best captured by first meshing these features with very small elements, that are then agglomerated together to create larger elements and obtain a final mesh of reasonable size.

A key aspect to the analysis in this paper is combining the unknowns on each face into a single boundary term, in a similar way as in \cite{brenner.sung:2018:virtual} for the VEM. This combination is only for the purpose of analysis (the boundary unknowns remain discontinuous polynomials, as standard in HHO), but it allows for the use of boundary trace inequalities that hold independently of the individual faces. Contrary to \cite{brenner.sung:2018:virtual}, however, we do not require the elements to be star-shaped. The analysis is presented on an anisotropic diffusion problem but, thanks to the generic results we establish, can easily be extended to other models and in particular problems involving non-linear operators such as the $p$-Laplacian as in \cite[Section 6]{di-pietro.droniou:2020:hybrid}. The HHO scheme is given in Section \ref{sec:discrete.problem} with assumptions on the stabilisation term less restrictive than standard (see Remark \ref{rem:stab.assumption}). This change in design condition removes in particular a continuity assumption on the HHO local bilinear form. Previous approaches for the HHO analysis used this continuity together with a stability property of the local interpolator to establish consistency estimates \cite{di-pietro.droniou:2020:hybrid}; this approach however led to estimates that are not robust with respect to the number or relative sizes of faces in each element. Our analysis therefore differs from the typical approach due to weaker assumptions on the stabilsation form. In Section \ref{sec:error.estimates} we provide stability-independent error bounds in weighted seminorms. These error estimates scale linearly with the anisotropy of the diffusion tensor and have a reduced dependency on the diffusion tensor compared to previous estimates for HHO methods, such as in \cite[Theorem 3.18]{di-pietro.droniou:2020:hybrid}. In Section \ref{sec:analysis} we state some preliminary lemmas, and prove the results given in the previous section. Section \ref{sec:examples.sT} provides a number of potential stabilisation terms satisfying the required assumptions. In particular, it is established that -- upon a change of scaling -- the original HHO stabilisation satisfies our modified set of design assumptions. Finally, we conclude the paper with numerical results in 2 and 3 dimensions.

\subsection{Model and Assumptions on the Mesh}

We take a polytopal domain \(\Omega\subset \R^d\), \(d\ge 2\), and consider the Dirichlet problem: Find \(u\) such that 
\begin{align}
	-\nabla\cdot(\matK\nabla u) \eq f \quad\text{in}\quad\Omega\nl
	u \eq 0 \quad\text{on}\quad\partial\Omega, \label{eq:strong.form}
\end{align}
for some source term \(f\in \LTWO(\Omega)\) and diffusion tensor \(\matK\) assumed to be a symmetric, piecewise constant matrix-valued function satisfying, for all \(\bmx\in\R^d\),
\begin{equation}\label{eq:uniform.elliptic}
	\ulK\bmx\cdot\bmx \le (\matK\bmx)\cdot\bmx \le \olK\bmx\cdot\bmx 
\end{equation}
for two fixed real numbers \(0 < \ulK \le \olK\).
The variational problem reads: find \(u\in \HONEzr( \Omega) \) such that
\begin{equation}\label{eq:weak.form}
	\a(u, v)  = \ell(v), \qquad\forall v\in \HONEzr(\Omega), 
\end{equation}
where \(\a(u, v) \defeq \brac[\Omega]{\matK\nabla u, \nabla v}\) and \(\ell(v) \defeq \brac[\Omega]{f, v}\). Here and in the following, \(\brac[X]{\cdot, \cdot}\) is the \(\LTWO\)-inner product of scalar- or vector-valued functions on a set \(X\) for its natural measure. Note that we also use \(\brac[X]{v, w}\) to denote the integral of the product \(vw\) whenever this product is integrable over \(X\) (which does not necessarily requires \(v,w\in \LTWO(X)\); \(\brac[X]{v, w}\) also makes sense if \(v\in L^1(X)\) and \(w\in L^\infty(X)\) for example).

Let \(\calH\subset(0, \infty)\) be a countable set of mesh sizes with a unique cluster point at \(0\). For each \(h\in\calH\), we partition the domain \(\Omega\) into a mesh \(\Mh=(\Th, \Fh)\), for which a detailed definition can be found in \cite[Definition 1.4]{di-pietro.droniou:2020:hybrid}. The set of mesh elements \(\Th\) is a disjoint set of polytopes such that \(\ol{\Omega}=\bigcup_{T\in\Th}\ol{T}\). The set \(\Fh\) is a collection of mesh faces forming a partition of the mesh skeleton, i.e. \(\bigcup_{T\in\Th}\bdryT=\bigcup_{F\in\Fh}\ol{F}\). The boundary faces \(F\subset\partial \Omega\) are gathered in the set \(\Fhb\). The parameter \(h\) is given by \(h\defeq\max_{T\in\Th}\hT\) where, for \(X=T\in\Th\) or \(X=F\in\Fh\), \(\hX\) denotes the diameter of \(X\). We shall also collect the set of faces attached to an element \(T\in\Th\) in the set \(\FT:=\{F\in\Fh:F\subset T\}\). The (constant) unit normal to \(F\in\FT\) pointing outside \(T\) is denoted by \(\norTF\), and \(\norT:\bdryT\to\R^d\) is the piecewise constant outer unit normal defined by \((\norT)|_F=\norTF\) for all \(F\in\FT\).

The regularity assumption in \cite[Definition 1.9]{di-pietro.droniou:2020:hybrid} on sequences of meshes \((\Mh)_{h\in\calH}\) forces each face to have a comparable (uniformly in \(h\)) size to the cells it belongs to, and imposes also a uniform upper bound on \(\CARD{\FT}\); this prevents for example from considering meshes obtained by coarsening fine meshes. The following assumption, made in the rest of this paper, is much less restrictive -- for example and contrary to \cite[Definition 1.9]{di-pietro.droniou:2020:hybrid}, it covers the mesh family represented in Figure \ref{fig:meshA}.

\begin{assumption}[Regular mesh sequence]\label{assum:star.shaped}	
	There exists a constant \(\varrho>0\) such that, for each \(h\in\calH\), each \(T\in\Th\) and each \(F\in\Fh\) is connected by star-shaped sets with parameter \(\varrho\) (see \cite[Definition 1.41]{di-pietro.droniou:2020:hybrid}). 
\end{assumption}	

\begin{remark}[Connected by star-shaped set] 
	The assumption of connectedness by star-shaped sets with parameter \(\varrho\)  means that each cell/face \(X\) can be written as the union of less that \(\varrho^{-1}\) sets \(X_i\), each one being star-shaped with respect to all points in a ball of radius \(\varrho h_{X_i}\), and that between any two sets \(X_i, X_j\) we can find a path of sets whose pairwise intersection contains balls of radius \(\varrho \hX\). 
\end{remark}

\begin{remark}[Assumption on the faces] 
	The requirement that each face is connected by star-shaped sets is only required in Lemmas \ref{lem:piFT.bound} and \ref{lem:L2proj.bdry} where we state properties of projectors on faces in the \(\LP{p}\)-norm. These properties rely on inverse Lebesgue inequalities, which in turn require the condition on the faces. When considering linear models, the properties are only required for \(p=2\) (which do not require inverse Lebesgue inequalities) and thus the assumption that the faces are connected by star-shaped sets can be dropped.
\end{remark}

We further require that the elements of the mesh align with the discontinuities of the diffusion tensor, i.e., for each \(T\in\Th\), \(\matK|_T\defeq \matKT\) is a constant matrix. In an analogous manner to \eqref{eq:uniform.elliptic} we define quantities \(0 < \ulKT \le \olKT\) to satisfy
\[
	\ulKT\bmx\cdot\bmx \le (\matKT\bmx )\cdot\bmx \le \olKT\bmx\cdot\bmx \qquad\forall\bmx\in\R^d.
\]
The diffusion anisotropy ratio \(\alphaT \defeq \frac{\olKT}{\ulKT}\) also comes of use.


\section{The Discrete Problem}\label{sec:discrete.problem}
Hybrid High-Order methods hinge on the local approximation of the variational problem on each element \(T\in\Th\). This is achieved here similar to the procedure found in \cite{di-pietro.droniou:2020:hybrid}, however, the analysis is performed differently to cope with the presence of many small faces. We begin by recalling the definition of local polynomials spaces and associated projectors. We then move on to presenting the local space of functions on each element, given by a couple \(( u,v) \) where \(u\) is a polynomial on the element and \(v\) is a piecewise discontinuous polynomial function on the boundary. Finally, we present the HHO method and state the error estimates in various norms, with constants independent of the number or relative size of the faces in each element.

In the following, we fix \(k,l\in\N\) such that \(|k-l| \le 1\); these correspond to the polynomial degree of the face and element unknowns of the HHO method. From hereon, we shall denote \(f\lesssim g\) to mean \(f \le Cg\) where \(C\) is a constant depending only on \(\Omega\), \(d\), \(\varrho\), \(k\) and \(l\), but independent of the considered face/element and quantities \(f,g\). We shall also write \(f\approx g\) if \(f\lesssim g\) and \(g\lesssim f\). When necessary, we make some additional dependencies of the constant \(C\) explicit.

\begin{remark}[Small and numerous faces]\label{rem:small.faces}
	The regularity parameter \(\varrho\) is not impacted by the existence of faces \(F\in\FT\), in some elements \(T\), whose diameter \(\hF\) is much smaller than \(\hT\), or by \(\max_{T\in\Th}\CARD{\FT}\). This carries out to the hidden constants in \(\lesssim\) and means that all our estimates are valid even for meshes with (possibly) many small faces in otherwise relatively ``round'' elements.
\end{remark}

\subsection{Polynomial Spaces and Projectors}

Let \(X=T\in\Th\) or \(X=F\in\Fh\) be a face or an element in a mesh \(\Mh\), and let \(\POLY{l}(X)\) be the set of \(d_X\)-variate polynomials of degree \(\le l\) on \(X\), where \(d_X\) is the dimension of \(X\). We denote by \(\piX{0, l}:L^1(X)\to\POLY{l}(X)\) the \(\LTWO\)-orthogonal projector \cite[Section 1.3]{di-pietro.droniou:2020:hybrid} of order \(l\). It is defined by: for all \(v\in L^1(X)\) and \(w\in\POLY{l}(X)\),
\[
	\brac[X]{\piX{0, l}v,w} = \brac[X]{v, w}. 
\]

The space of piecewise discontinuous polynomial functions on an element boundary is given by
\begin{equation}\label{eq:bdry.space.def}
	\POLY{k}(\FT) \defeq \{v\in L^1(\bdryT):v|_F\in\POLY{k}(F)\quad\forall F\in\FT\}.
\end{equation}
The \(\LTWO\) orthogonal projector on an element boundary \(\piFTzr{k}:L^1(\bdryT)\to\POLY{k}(\FT)\) is then defined to satisfy \(\piFTzr{k}v|_F=\piFzr{k}v\) for all \(v\in L^1(\bdryT)\) and \(F\in\FT\). Alternatively, we could have defined \(\piFTzr{k}\) to be the unique element of \(\POLY{k}(\FT)\) that satisfies
\begin{equation}\label{eq:piFT.orthogonality}
	\brac[\bdryT]{v-\piFTzr{k}v, w} = 0 \qquad\forall w\in\POLY{k}(\FT). 
\end{equation}
In particular, \eqref{eq:piFT.orthogonality} allows us to replace \(v\) by \(\piFTzr{k}v\) whenever \(v\) occurs in an inner-product with a polynomial \(w\in\POLY{k}(\FT)\).

We require to define some weighted inner-products and norms to account for the diffusion tensor \(\matK\). For an element boundary \(\bdryT\), the weighted inner-product \(\brac[\matK, \bdryT]{\cdot, \cdot}:\LTWO(\bdryT)\times\LTWO(\bdryT)\to\R\) is defined for all \(v,w\in \LTWO(\bdryT)\) via
\begin{equation}\label{eq:ip.weighted.bdry}
	\brac[\matK, \bdryT]{v, w} \defeq \brac[\bdryT]{\matKT^{\frac12}\norT\,v,\matKT^{\frac12}\norT\,w}=\brac[\bdryT]{[\matKT\norT\cdot\norT]v, w}.
\end{equation}
For all \(r\ge 1\) and \(v\in \HS{r}(T)\) the weighted \(\HS{r}\)-seminorm \(\seminorm[\matK, \HS{r}(T)]{{\cdot}}\) is defined as
\begin{equation}\label{eq:norm.weighted.hr}
	\seminorm[\matK, \HS{r}(T)]{v} \defeq \seminorm[H^{r-1}(T)^d]{\matKT^{\frac12}\nabla v}. 
\end{equation}
We note the following norm equivalences:
\begin{align}
	\ulKT\norm[\bdryT]{v}^2 \le \norm[\matK, \bdryT]{v}^2 \le \olKT\norm[\bdryT]{v}^2 &\qquad
	\forall v\in \LTWO(\bdryT), \label{eq:norm.equiv.bdry}\\
	\ulKT\seminorm[\HS{r}(T)]{v}^2 \le \seminorm[\matK,\HS{r}(T)]{v}^2 \le \olKT\seminorm[\HS{r}(T)]{v}^2 &\qquad
	\forall v\in \HS{r}(T). \label{eq:norm.equiv.hr}
\end{align}
To properly handle the diffusion tensor in the problem considered, we use the weighted/oblique elliptic projector \cite[Section 3.1.2]{di-pietro.droniou:2020:hybrid} defined by, for all \(T\in\Th\), as \(\piKTe{k+1}:\WSP{1,1}(T) \to\POLY{k+1}(T) \) such that, for all \(v\in \WSP{1,1}(T)\),
\begin{align*}
	\brac[T]{\matKT\nabla \brac{v-\piKTe{k+1}v}, \nabla w} \eq 0\qquad\forall w\in\POLY{k+1}(T), \\
	\brac[T]{v-\piKTe{k+1}v,1} \eq 0. 
\end{align*}

\subsection{Local HHO Space}

For each element \(T\in\Th\), the local space of unknowns is defined as
\begin{equation}\label{eq:local.hho.space}
	\UTkl\defeq \POLY{l}(T) \times \POLY{k}(\FT). 
\end{equation}
To avoid excessive notation, when referring to the boundary term of \(\ulvT=(\vT,\vFT)\in\UTkl\) restricted to a face \(F\in\FT\), we shall simply write \(\vF\defeq \vFT|_F\). We also endow the space \(\UTkl\) with the seminorm \(\seminorm[1,\matK,\bdryT]{{\cdot}}:\UTkl\to\R\) defined for all \(\ulvT\in\UTkl\) as
\begin{equation}\label{eq:discrete.norm.def}
	\seminorm[1,\matK,\bdryT]{\ulvT} \defeq \hT^{-\frac12}\norm[\matK,\bdryT]{\vFT - \vT}. 
\end{equation}
\begin{remark}[The case \((k,l)=(0,-1)\)]
	For the case where the face unknowns are constant functions (\(k=0\)), it is possible to define an HHO method with \(l=-1\), corresponding to zero degrees of freedom on each element. This is achieved by defining the element term as a weighted sum of the face terms. We do not cover this case here, and refer the interested reader to \cite[Section 5.1]{di-pietro.droniou:2020:hybrid}.
\end{remark}
The local interpolator \(\ITkl:\HONE(T) \to\UTkl\) is defined for all \(v\in \HONE(T)\) as \(\ITkl v\defeq(\piTzr{l}v, \piFTzr{k}v)\).
On each element we locally reconstruct a potential from the space of unknowns via the operator \(\pKT{k+1}:\UTkl\to\POLY{k+1}(T)\) defined to satisfy, for all \(\ulvT\in\UTkl\),
\begin{align}
	\brac[T]{\matKT\nabla\pKT{k+1}\ulvT, \nabla w} \eq -\brac[T]{\vT, \nabla\cdot(\matKT \nabla w)} + \brac[\bdryT]{\vFT, \matKT\nabla w \cdot\norT} \qquad\forall w\in\POLY{k+1}(T), \label{eq:pKT.def} \\
	\brac[T]{\vT - \pKT{k+1}\ulvT, 1} \eq 0. \label{eq:pKT.closure}
\end{align}
We note that \(\pKT{k+1}\circ\ITkl=\piKTe{k+1}\) (see \cite[Eq. (3.24)]{di-pietro.droniou:2020:hybrid}).
This potential reconstruction allows us to approximate \(\a(u,v) \) on each element by the bilinear form \(\aKT:\UTkl\times\UTkl\to\R\) defined as
\begin{equation}\label{eq:local.form.def}
	\aKT(\uluT,\ulvT) \defeq \brac[T]{\matKT\nabla \pKT{k+1}\uluT, \nabla \pKT{k+1}\ulvT} + \sKT(\uluT, \ulvT), 
\end{equation}
where \(\sKT:\UTkl\times\UTkl\to\R\) is a local stabilisation term such that the following assumptions hold.

\begin{assumption}[Local stabilisation term]\label{assum:stability}
	The stabilisation term \(\sKT\) is a symmetric, positive semi-definite bilinear form that satisfies:
	\begin{enumerate}
		\item \emph{Coercivity.} For all \(\ulvT\in\UTkl\) it holds that
		\begin{equation}\label{eq:stab.coercivity}
			\seminorm[1,\matK,\bdryT]{\ulvT}^2 \lesssim \alphaT \aKT(\ulvT,\ulvT). 
		\end{equation}
		\item \emph{Consistency for smooth functions}. For all \(r\in\{0,\dots,k\}\) and \(v\in\HS{r+2}(T)\) it holds that
		\begin{equation}\label{eq:stab.consistency}
			\sKT(\ITkl v, \ITkl v) \lesssim \locerrorRHS{v}.
		\end{equation}
	\end{enumerate}
\end{assumption}

Examples of stabilisation forms satisfying Assumption \ref{assum:stability} are given in Section \ref{sec:examples.sT}.

\begin{remark}[Assumption on the stabilisation term]\label{rem:stab.assumption}
	We note here that the coercivity assumption \eqref{eq:stab.coercivity} on the stability \(\sKT\) is less restrictive than that given in \cite[Assumption 2.4]{di-pietro.droniou:2020:hybrid}. There, it is assumed that \(\aKT\) is coercive and continuous with respect to the seminorm defined by 
	\[
	\norm[1,\matK,T]{\ulvT}^2 \defeq \sum_{F\in\FT}\frac{\matKT\norTF\cdot\norTF}{\hF}\norm[F]{\vF-\vT}^2 + \seminorm[\matK,\HONE(T)]{\vT}^2,\nn
	\] 
	which for small faces can be significantly larger than the seminorm defined by \eqref{eq:discrete.norm.def}. We additionally remove the requirement that \(\aKT\) is continuous with respect to \(\norm[1,\matK,T]{{\cdot}}\). These weakened assumptions are key to obtaining error estimates that are independent of the number or smallness of faces in each element. The relaxed continuity assumption requires us to force consistency of the stabilisation form for smooth functions, which is a stricter assumption than polynomial consistency alone (combined with a continuity assumption, this consistency for smooth functions is equivalent to polynomial consistency \cite[Proposition 2.14]{di-pietro.droniou:2020:hybrid}). These weakened assumptions on \(\aKT\) are particularly useful when considering enriched schemes for which continuity is not guaranteed \cite{yemm:2021:xhho}.
\end{remark}

\subsection{Global Space and HHO Scheme}

The global space of unknowns is defined as 
\begin{equation}\label{eq:global.space.def}
	\Uhkl\defeq \Big\{\ulvh=((\vT)_{T\in\Th},(\vF)_{F\in\Fh})\,:\,\vT\in\POLY{l}(T)\quad\forall T\in\Th\,,
\vF\in\POLY{k}(F)\quad\forall F\in\Fh \Big\}.
\end{equation}
To account for the homogeneous boundary conditions, the following subspace is also introduced,
\begin{equation}\label{eq:global.space.hom.def}
	\Uhklzr\defeq\{\ulvh \in\Uhkl:v_{F}=0\quad\forall F\in\Fhb\}.
\end{equation}
For any \(\ulvh\in\Uhkl\) we denote its restriction to an element \(T\) by \(\ulvT=(\vT,\vFT)\in\UTkl\) (where, naturally, \(\vFT\) is defined form \((\vF)_{F\in\FT}\)). We also denote by \(\vh\) the piecewise polynomial function satisfying \(\vh|_T=\vT\) for all \(T\in\Th\).

The space of piecewise \(\HONE\) functions is defined as \(\HONE(\Th)\defeq\{\phi\in \LTWO(\Omega)\,:\nabla_h\phi\in \LTWO(\Omega)\}\), where \(\nabla_h\) denotes the broken gradient satisfying \((\nabla_h\phi)|_T = \nabla(\phi|_T)\) for all \(T\in\Th\). We endow the space \(\HONE(\Th)\) with the weighted seminorm
\[
	\seminorm[\matK,\HONE(\Th)]{v} \defeq \norm[\Omega]{\matK^\frac12\nabla_h v}. 
\]

The global operators \(\pKh{k+1}:\Uhkl\to\POLY{k+1}(\Th)\), \(\piKhe{k+1}:\HONE(\Th)\to\POLY{k+1}(\Th)\), and  \(\pihzr{l}:L^1(\Omega)\to\POLY{l}(\Th)\) are defined such that their actions restricted to an element \(T\in\Th\) are that of \(\pKT{k+1}\), \(\piKTe{k+1}\), and \(\piTzr{l}\) respectively. The global interpolator \(\Ihkl:\HONE(\Omega)\to\Uhkl\) is defined as \(\Ihkl v \defeq ((\piTzr{l}v)_{T\in\Th},(\piFzr{k}v)_{F\in\Fh})\). It follows that \(\pKh{k+1}\circ\Ihkl v = \piKhe{k+1}v\) for all \(v\in\HONE(\Omega)\).

The global bilinear forms \(\aKh:\Uhkl\times\Uhkl\to\mathbb{R}\) and \(\sKh:\Uhkl\times\Uhkl\to\mathbb{R}\) are defined as
\[
	\aKh(\uluh, \ulvh) \defeq \sum_{T\in\Th} \aKT(\uluT,\ulvT)
	\quad\textrm{and}\quad
	\sKh(\uluh, \ulvh) \defeq \sum_{T\in\Th} \sKT(\uluT,\ulvT).
\]
We also define the discrete energy norm \(\norm[\aKh]{{\cdot}}\) on \(\Uhklzr\) as
\begin{equation}\label{eq:energy.norm.def}
	\norm[\aKh]{\ulvh}\defeq \aKh(\ulvh,\ulvh)^\frac{1}{2} \qquad \forall \ulvh\in\Uhkl.
\end{equation}	
The HHO scheme reads: find \(\uluh\in\Uhklzr\) such that
\begin{equation}\label{eq:discrete.problem}
	\aKh(\uluh, \ulvh) = \ell_h(\ulvh) \qquad\forall \ulvh\in\Uhklzr, 
\end{equation}
where \(\ell_h:\Uhklzr\to\R\) is a linear form defined as
\begin{equation}\label{eq:discrete.src.term}
	\ell_h(\ulvh) \defeq \sum_{T\in\Th}\brac[T]{f,\vT}. 
\end{equation}

\subsection{Error Estimates}\label{sec:error.estimates}

We state here the error estimates, in various norms, that we will prove on the HHO scheme described above. In all these estimates, the hidden constants are robust with respect to the number or relative sizes of faces in each elements (see Remark \ref{rem:small.faces}). The first error estimate is given in discrete and continuous energy norms.

\begin{theorem}[Energy error]\label{thm:energy.error}
	Let \(r\in\{0,\dots,k\}\), \(u\in\HONEzr(\Omega)\cap\HS{r+2}(\Th)\) be the exact solution to the continuous problem \eqref{eq:weak.form}, and \(\uluh\in\Uhklzr\) be the solution to the HHO scheme \eqref{eq:discrete.problem}. The following energy error estimates hold:
	\begin{align}
		\norm[\aKh]{\uluh - \Ihkl u} \les \errorRHS{u}, \label{eq:discrete.energy}\\
		\seminorm[\matK,\HONE(\mathcal{T}_h)]{\pKh{k+1}\uluh - u} \les \errorRHS{u}. \label{eq:continuous.energy}
	\end{align}
\end{theorem}	

\begin{remark}[Diffusion weighted error estimates]
	The error estimates in Theorem \ref{thm:energy.error} are an improved version of those found in \cite[Section 3.1]{di-pietro.droniou:2020:hybrid}. Specifically, in this reference, each term \(\seminorm[\matK,\HS{r+2}(T)]{u}\) is replaced by \(\olKT^{1/2}\seminorm[\HS{r+2}(T)]{u}\), which are larger. This difference can be significant in practice. If \(\matKT\) is strongly anisotropic, then the solution \(u\) to \eqref{eq:strong.form} is expected to vary much less in directions of stronger diffusion. As a consequence, \(\olKT^{1/2}\seminorm[\HS{r+2}(T)]{u}\) could be much larger than \(\seminorm[\matK,\HS{r+2}(T)]{u}\) since, in the latter term, the large eigenvalues of \(\matKT\) could multiply small directional gradients of \(u\).
\end{remark}

The second error estimate concerns the jumps across faces of the reconstructed potentials; this estimate indicates that, as \(h\to 0\), these potentials become ``more and more'' conforming. To formalise this we first define the jump operator \(\jump[F]{\cdot}\) for all \(v\in \HONE(\Th)\) via
\[
	\jump[F]{v} \defeq (v|_{T_1})|_F - (v|_{T_2}) |_F,
\]
where \(\{T_1 , T_2\} =: \Th[F]\) are the two cells on each side of \(F\in\Fh\backslash\Fhb\); for boundary faces \(F\in\Fhb\), recalling that we are working with homogeneous Dirichlet boundary conditions we set, with \(T\in\Th\) such that \(F\in\FT\),
\[
	\jump[F]{v} \defeq (v|_T)|_F.
\]
We shall also define \(\jump[\bdryT]{\cdot}\) such that \(\jump[\bdryT]{v}|_F =\jump[F]{v}\).

\begin{theorem}[Convergence of the jumps]\label{thm:jump.error}
	Let \(r\in\{0,\dots,k\}\), \(u\in \HONEzr(\Omega) \cap \HS{r+2}(\Th)\) be the exact solution to the continuous problem \eqref{eq:weak.form}, and \(\uluh\in\Uhklzr\) be the solution to the HHO scheme \eqref{eq:discrete.problem}. The following jump error estimate holds:
	\begin{equation}\label{eq:jump.error}
		\BRAC[]{\sum_{T\in\Th} \frac{\ulKT}{\alphaT\hT}\norm[\bdryT]{\jump[\bdryT]{\pKh{k+1}\uluh}}^2}^\frac12 \lesssim \errorRHS{u}. 
	\end{equation}
\end{theorem}

\begin{remark}[Stability-independent jump estimate]
	The form of Theorem \ref{thm:jump.error} improves upon that found in \cite[Section 2.3.2]{di-pietro.droniou:2020:hybrid}. In particular, the error estimate is independent of the choice of stabilisation term, and the dependency on the diffusion tensor \(\matK\) is tracked.
\end{remark}

We finally turn to an estimate for the error induced under the \(\LTWO\)-norm. As is seen in \cite{di-pietro.droniou:2020:hybrid}, the convergence rates are found to be optimal only when the diffusion tensor is constant, and the problem is posed on a convex domain to ensure the elliptic regularity of the model. For this reason, we do not attempt to precisely track the dependency of the hidden constants with respect to the diffusion tensor.

\begin{theorem}[\(\LTWO\) error]\label{thm:l2.error}
	Let \(k,l\ge 1\), \(r\in\{1,\dots,k\}\), \(u\in \HONEzr(\Omega) \cap \HS{r+2}(\Th)\) be the exact solution to the continuous problem \eqref{eq:weak.form}, and \(\uluh\in\Uhklzr\) be the solution to the HHO scheme \eqref{eq:discrete.problem}. For any convex domain \(\Omega\), and constant diffusion tensor \(\matK\), the \(\LTWO\)-error satisfies the estimate
	\begin{equation}\label{eq:l2.error}
		\norm[\Omega]{\pKh{k+1}\uluh-u} \lesssim h^{r+2}\seminorm[\HS{r+2}(\Th)]{u}, 
	\end{equation}
	where the hidden constant additionally depends on the diffusion tensor \(\matK\).
\end{theorem}

\begin{remark}[The cases \(k,l<1\)]
	For the case \(k=0\), an error estimate in \(\LTWO\)-norm converging at the improved rate of \(h^2\) can be obtained \cite[Theorem 5.16]{di-pietro.droniou:2020:hybrid}. When \((k,l)=(1,0)\), numerical tests show that no improved rate of convergence in \(\LTWO\)-error can be expected compared to the \(h^2\) rate of convergence in energy norm.
\end{remark}

\section{Error Analysis}\label{sec:analysis}

We prove here the error estimates stated above, starting first with some preliminaries which consist in ensuring that
certain general inequalities, key to our analysis, are indeed robust with respect to the face sizes.

\subsection{Preliminary Results}\label{sec:prelim}

\subsubsection{Lebesgue, Sobolev and Trace Inequalities}

Let \(X\) be a face or an element in a mesh \(\Mh\). Under Assumption \ref{assum:star.shaped}, the inradius of \(X\) is equivalent (uniformly in \(h\)) to the diameter of \(X\) and thus, by \cite[Lemma 1.25]{di-pietro.droniou:2020:hybrid}, for all \((p_1,p_2)\in [1,\infty]\) the following direct and reverse Lebesgue inequality holds, with hidden multiplicative constant depending additionally on \(p_1\) and \(p_2\):
\begin{equation*}
	\norm[L^{p_1}(X)]{v} \approx |X|^{\frac{1}{p_1}-\frac{1}{p_2}}\norm[L^{p_2}(X)]{v}\qquad\forall v\in\POLY{l}(X).
\end{equation*}	
In the relation above, \(|X|\) denotes the \(d_X\)-dimensional measure of \(X\).

The following inverse Sobolev embedding is also highly relevant to the analysis required in this paper. We denote by \(\seminorm[\WSP{s,p}(X)]{{\cdot}}\) the \(\LP{p}\) norm of the \(s\) distributional derivative, and take \(m\le s\). It holds, by \cite[Corollary 1.29]{di-pietro.droniou:2020:hybrid}, that

\begin{equation}\label{eq:inverse.sobolev}
	\seminorm[\WSP{s,p}(X)]{v} \lesssim \hX^{m-s}\seminorm[\WSP{m,p}(X)]{v}\qquad\forall v\in\POLY{l}(X), 
\end{equation}
with hidden multiplicative constant depending additionally on \(m\), \(s\), and \(p\).

The following continuous trace inequality has been established in \cite[Section 2.6]{brenner.sung:2018:virtual} for sets that are star-shaped with respect to balls of radius comparable to the
set diameter. An extension to sets connected by star-shaped sets is not difficult, and provided for the sake of completeness.

\begin{lemma}[Continuous trace inequality]
	For all \(p\in [1,\infty)\) and \(v\in \WSP{1,p}(T)\), it holds
	\begin{equation}\label{eq:continuous.trace}
		\hT\norm[\LP{p}(\bdryT)]{v}^p \lesssim \norm[\LP{p}(T)]{v}^p+\hT^p\norm[\LP{p}(T)]{\nabla v}^p, 
	\end{equation}
	where the hidden constant depends additionally on \(p\), and the space \(\LP{p}(\bdryT)\) is endowed with the norm
	\[
		\norm[\LP{p}(\bdryT)]{v}:=\left(\sum_{F\in\FT}\norm[\LP{p}(F)]{v}^p\right)^{1/p}. 
	\]
\end{lemma}

\begin{proof}
	By assumption, \(T=\cup_{i=1}^N X_i\) with \(N\le \varrho^{-1}\), each \(X_i\) being star-shaped with respect to a ball of radius \(\gtrsim h_{X_i}\), and \(h_{X_i}\approx \hT\). Applying the trace inequality of \cite[Section 2.6]{brenner.sung:2018:virtual} to each \(X_i\) yields
	\begin{equation*}
		h_{X_i}\norm[\LP{p}(\partial X_i)]{v}^p \lesssim \norm[\LP{p}(X_i)]{v}^p+h_{X_i}^p\norm[\LP{p}(X_i)]{\nabla v}^p. 
	\end{equation*}
	Noticing that \(\bdryT\subset \cup_{i=1}^N\partial X_i\), we use \(h_{X_i}\approx \hT\) and sum the above inequality over \(i=1,\ldots,N\) to get
	\[
		\hT\norm[\LP{p}(\bdryT)]{v}^p \lesssim \sum_{i=1}^N \norm[\LP{p}(X_i)]{v}^p + \hT^p\sum_{i=1}^N\norm[\LP{p}(X_i)]{\nabla v}^p
		\lesssim N \norm[\LP{p}(T)]{v}^p + \hT^p N\norm[\LP{p}(T)]{\nabla v}^p,
	\]
	where the second inequality follows since each \(X_i\) is contained in \(T\). The proof is concluded by recalling that \(N\le\varrho^{-1}\).
\end{proof}

By combining \eqref{eq:inverse.sobolev} and \eqref{eq:continuous.trace} with \(s=1\) and \(m=0\), we have for all \(v\in\POLY{l}( T) \) the following discrete trace inequality, in which the hidden constant depends additionally on \(p\):
\begin{equation} \label{eq:discrete.trace}
	\hT\norm[\LP{p}(\bdryT)]{v}^p \lesssim \norm[\LP{p}(T)]{v}^p.
\end{equation}

\begin{remark}
	In \cite{cangiani.dong:2019:version} a sharp discrete trace inequality for the case $p=2$ is shown with fully explicit constant, and weaker assumptions on the element $T$. The same reference also provides explicit constants for the $\HONE$ to $\LTWO$ inverse inequality. 
	
	The trace inequality \eqref{eq:discrete.trace}, with a constant that does not depend on the relative sizes or number of faces in $T$, is an essential tool for establishing \eqref{eq:stab.coercivity} and \eqref{eq:stab.consistency} for specific stabilisation bilinear forms. In this respect, our analysis is therefore based on the same basic tools as used for DG methods in \cite{cangiani.dong:2019:version}. However, because of the presence of the higher-order potential reconstruction $\pKT{k+1}$, and the need to handle two different local seminorms (namely, $\seminorm[1,\matK,\bdryT]{\cdot}$ and $\aKT(\cdot,\cdot)^{\frac12}$), the HHO analysis differs in several aspects from the DG analysis.
\end{remark}

\subsubsection{Projectors on Polynomial Spaces}

The following properties of \(\piX{0,l}\) are taken from \cite[Section 1.3]{di-pietro.droniou:2020:hybrid}.

\begin{lemma}[Approximation and boundedness properties of the \(\LTWO\)-projector on a face/element]
	For all \(s\in\{0,\dots,l+1\}\), \(m\in\{0,\dots,s\}\), \(p\ge 1\) and \(v\in \WSP{s,p}(X)\), the \(\LTWO\)-orthogonal projector satisfies:
	\begin{equation}\label{eq:piTzr.approx}
		\seminorm[\WSP{m,p}(X)]{v-\piX{0,l}v} \lesssim \hX^{s-m}\seminorm[\WSP{s,p}(X)]{v}\quad\textrm{and}\quad
		\seminorm[\WSP{m,p}(X)]{\piX{0,l} v} \lesssim \seminorm[\WSP{m,p}(X)]{v},
	\end{equation}
	where the hidden constant additionally depend on \(m\), \(s\), and \(p\).
\end{lemma}

\begin{lemma}[Boundedness of the \(\LTWO\)-orthogonal projector on an element boundary]\label{lem:piFT.bound}
	Let \(p\ge 1\). Then it holds, with hidden multiplicative constant depending additionally on \(p\), that
	\begin{equation}\label{eq:piFT.bound}
		\norm[\LP{p}(\bdryT)]{\piFTzr{k}v} \lesssim \norm[\LP{p}(\bdryT)]{v}\qquad\forall v\in \LP{p}(\bdryT). 
	\end{equation}
\end{lemma}	

\begin{proof}
	Simply raise the bound in \eqref{eq:piTzr.approx} with \(X=F\in\FT\) and \(m=0\) to the power \(p\) and sum over the faces \(F\in\FT\) (or take the maximum over the faces in the case \(p=+\infty\)).	
\end{proof}

\begin{lemma}[Properties of the \(\LTWO\)-orthogonal projector on an element boun\-dary]\label{lem:L2proj.bdry}
	Let \(p\ge 1\) and \(s\in\{1,\dots,\min(k,l)+1\}\). Then it holds, with hidden multiplicative constant depending additionally on \(p\) and \(s\), that
	\begin{equation}\label{eq:piFT.approx}
		\hT^{\frac{1}{p}}\norm[\LP{p}(\bdryT)]{\piFTzr{k}v-\piTzr{l}v} \lesssim \hT^s\seminorm[\WSP{s,p}(T)]{v}\qquad\forall v\in \WSP{s,p}(T). 
	\end{equation}
\end{lemma}	
	
\begin{proof}
	We first consider the case \(k \ge l\). By the polynomial consistency of projectors \cite[Proposition 1.35]{di-pietro.droniou:2020:hybrid}, it holds that \((\piTzr{l}v)|_{\bdryT}=\piFTzr{k}(\piTzr{l}v)|_{\bdryT}\) and thus, using the boundedness \eqref{eq:piFT.bound} of \(\piFTzr{k}\) and the continuous trace inequality \eqref{eq:continuous.trace},
	\begin{align}
		\hT^{\frac1p}\norm[\LP{p}(\bdryT)]{\piFTzr{k}v-\piTzr{l}v} = \hT^{\frac1p}{}&\norm[\LP{p}(\bdryT)]{\piFTzr{k}(v-\piTzr{l}v)} 
		\lesssim \hT^{\frac1p}\norm[\LP{p}(\bdryT)]{v-\piTzr{l}v} \nl
		\les \norm[\LP{p}(T)]{v-\piTzr{l}v} + \hT\seminorm[\WSP{1,p}(T)]{v-\piTzr{l}v}.  \label{eq:piFT.approx.proof.1}
	\end{align}
	The conclusion follows from the approximation properties and boundedness (see \eqref{eq:piTzr.approx}) of the \(\LTWO\)-orthogonal projector. 
	
	Consider now \(k < l\). By a triangle inequality
	\begin{equation}\label{eq:piFT.approx.proof.2}
		\norm[\LP{p}(\bdryT)]{\piFTzr{k}v-\piTzr{l}v} \le \norm[\LP{p}(\bdryT)]{\piFTzr{k}v-\piTzr{k}v} + \norm[\LP{p}(\bdryT)]{\piTzr{k}v-\piTzr{l}v}. 
	\end{equation}
	The first term of \eqref{eq:piFT.approx.proof.2} is of the form covered by \eqref{eq:piFT.approx.proof.1}. The second term may be bounded using a discrete trace inequality \eqref{eq:discrete.trace}:
	\[
		\hT^{\frac1p}\norm[\LP{p}(\bdryT)]{\piTzr{k}v-\piTzr{l}v} \lesssim \norm[\LP{p}(T)]{\piTzr{k}v-\piTzr{l}v}
		\lesssim \norm[\LP{p}(T)]{\piTzr{k}v-v},
	\]
	where the conclusion follows from \(\piTzr{k}=\piTzr{l}\piTzr{k}\) (since \(l>k\)), and the boundedness \eqref{eq:piTzr.approx} of \(\piTzr{l}\).
	The conclusion then follows as before from \eqref{eq:piTzr.approx}.
\end{proof}

The next corollary is a consequence of the norm equivalences \eqref{eq:norm.equiv.bdry}--\eqref{eq:norm.equiv.hr} and Lemma \ref{lem:L2proj.bdry}.
\begin{corollary}[Properties of the \(\LTWO\)-orthogonal projector on an element boundary with respect to weighted norms]
	Let \(l\) and \(k\) be non-negative integers, and \(s\in\{1,\dots, \min(k,l) +1\}\). Then it holds, with hidden multiplicative constant depending additionally on \(s\), that
	\begin{equation}\label{eq:piFT.weighted.approx}
		\hT^{\frac{1}{2}}\norm[\matK,\bdryT]{\piFTzr{k}v-\piTzr{l}v} \lesssim \alphaT^\frac12\hT^s\seminorm[\matK,\HS{s}(T)]{v}\qquad\forall v\in \HS{s}(T). 
	\end{equation}
\end{corollary}

\begin{lemma}[Approximation properties of the oblique elliptic projector]\label{lem:piKTe.approx}
	For all \(s\in\{1,\dots,k+2\}\), \(m\in\{1,\dots,s\}\) and \(v\in \HS{s}(T) \) the oblique elliptic projector satisfies
	\begin{equation}\label{eq:piKTe.approx}
		\seminorm[\matK,\HS{m}(T)]{v-\piKTe{k+1}v} \lesssim \hT^{s-m}\seminorm[\matK,\HS{s}(T)]{v},
	\end{equation}
	where the hidden constant depends additionally on \(s\) and \(m\).
\end{lemma}

The proof of Lemma \ref{lem:piKTe.approx} is analogous to that for the unweighted elliptic projector found in \cite{di-pietro.droniou:2020:hybrid}. However, it relies on the approximation of averaged Taylor polynomials (see \cite[Section 4.1]{brenner.scott:2007:mathematical}) in weighted seminorms. Since these approximation properties are not standard, we detail the proof below. Let us start with preliminary results.

For any integer \(s\ge 1\), the averaged Taylor polynomial operator \(Q^s:C^{s-1}(T)\to\POLY{s-1}(T)\) defined in \cite[Definition 4.1.3]{brenner.scott:2007:mathematical} is linear. Therefore, the remainder operator \(R^s:C^{s-1}(T)\to C^{s-1}(T)\) defined by \(R^s v \defeq v - Q^s v\) is also linear. By \cite[Proposition 4.1.17]{brenner.scott:2007:mathematical}, for any \(\nu\in\N^d\) such that \(|\nu| \le s\), \(Q^s\) satisfies
\[
	D^\nu Q^sv = Q^{s-|\nu|}D^\nu v\qquad\forall v\in H^{|\nu|}(T),
\]
where \(D^\nu\) denotes the distributional derivative. In particular, we infer for all \(j\in\{1,\dots,d\}\)
\begin{equation}\label{eq:taylor.prop}
	\partial_j R^sv = R^{s-1} \partial_j v\qquad\forall v\in \HONE(T). 
\end{equation}
The remainder term \(R^sv\) satisfies for all \(m\in\{0,\dots,s\}\)
\begin{equation}\label{eq:taylor.approx}
	\seminorm[\HS{m}(T)]{R^sv} \lesssim \hT^{s-m}\seminorm[\HS{s}(T)]{v}\qquad\forall v\in \HS{s}(T), 
\end{equation}
with hidden constant depending additionally on \(m\) and \(s\). A proof of \eqref{eq:taylor.approx} is provided on star-shaped sets in \cite{brenner.scott:2007:mathematical} and extended to sets connected by star shaped sets in \cite[Theorem 1.50]{di-pietro.droniou:2020:hybrid}.

\begin{proof}[Proof of Lemma \ref{lem:piKTe.approx}]
	Let \(v\in \HS{s}(T)\) and \(m\in\{1,\ldots,s\}\). By the linearity of \(R^s\) and \eqref{eq:taylor.prop}, and since \(\matKT\) is constant, we extend the approximation property \eqref{eq:taylor.approx} (with \((m-1,s-1)\) instead of \((m,s)\)) to the \(\matK\)-weighted seminorm as follows:
	\begin{align}
		\seminorm[\matK,\HS{m}(T)]{R^sv} \eq \seminorm[\HS{m-1}(T)^d]{\matKT^\frac12 \nabla R^sv} \nl 
		\eq \BRAC{\sum_{i=1}^d \SEMINORM[\HS{m-1}(T)]{\sum_{j=1}^d \BRAC[ij]{ \matKT^\frac12} \partial_j R^sv}^2}^\frac12 \nl
		\eq \BRAC{\sum_{i=1}^d \SEMINORM[\HS{m-1}(T)]{R^{s-1}\sum_{j=1}^d \BRAC[ij]{ \matKT^\frac12} \partial_j v}^2}^\frac12 \nl
		\les \hT^{s-m}\BRAC{\sum_{i=1}^d \SEMINORM[H^{s-1}(T)]{\sum_{j=1}^d \BRAC[ij]{ \matKT^\frac12} \partial_j v}^2}^\frac12 
		= \hT^{s-m}\seminorm[\matK,\HS{s}(T)]{v}. \label{eq:piKTe.approx.proof.1}
	\end{align}
	By the polynomial consistency of \(Q^s\), we have \(v-\piKTe{k+1}v = R^sv-\piKTe{k+1}R^sv\)	for all \(s\le k+2\). Therefore a triangle inequality yields
	\[
		\seminorm[\matK,\HS{m}(T)]{v-\piKTe{k+1}v} \le \seminorm[\matK,\HS{m}(T)]{R^sv} + \seminorm[\matK,\HS{m}(T)]{\piKTe{k+1}R^sv}. 
	\]
	The term \(\seminorm[\matK,\HS{m}(T)]{R^sv}\) already satisfies the desired bound due to equation \eqref{eq:piKTe.approx.proof.1}. For the second term, use an inverse Sobolev inequality \eqref{eq:inverse.sobolev} to get 
	\begin{align}
		\seminorm[\matK,\HS{m}(T)]{\piKTe{k+1}R^sv} = \seminorm[\HS{m-1}(T)]{\matKT^\frac12\nabla\piKTe{k+1}R^sv} \les \hT^{1-m}\norm[T]{\matKT^\frac12\nabla\piKTe{k+1}R^sv} \nl
		\eq \hT^{1-m}\seminorm[\matK,\HONE(T)]{\piKTe{k+1}R^sv}. \nn
	\end{align} 
	From the definition of the oblique elliptic projector it holds that
	\[
		\seminorm[\matK,\HONE(T)]{\piKTe{k+1}R^sv}^2 = \brac[T]{\matKT\nabla\piKTe{k+1}R^sv,\nabla R^sv}. 
	\]
	We may then infer from a Cauchy--Schwarz inequality and equation \eqref{eq:piKTe.approx.proof.1} that
	\[
		\hT^{1-m}\seminorm[\matK,\HONE(T)]{\piKTe{k+1}R^sv} \le \hT^{1-m}\seminorm[\matK,\HONE(T)]{R^sv} 
		\lesssim \hT^{1-m}\hT^{s-1}\seminorm[\matK,\HS{s}(T)]{v} = \hT^{s-m}\seminorm[\matK,\HS{s}(T)]{v}, 
	\]
	which concludes the proof.
\end{proof}

\subsection{Proof of the error estimates}

For a Banach space \((L,\norm[L]{\cdot})\), the dual norm of a linear form \(g:L\to \R\) is defined as
\begin{equation}\label{eq:dual.norm.def}
	\norm[L^*]{g} \defeq \sup_{x\in L\backslash\{0\}}\frac{\seminorm{g(x)}}{\norm[L]{x}}. 
\end{equation}
We denote the dual norm on the Banach space \((\Uhklzr,\norm[\aKh]{{\cdot}})\) by \(\norm[\aKh,*]{{\cdot}}\). The following lemma gives an estimate on the consistency error, which is at the core of all the error estimates.

\begin{lemma}[Consistency error]
	The consistency error \(\calE_h( w;\cdot) :\Uhklzr\to\R\) is a linear form defined for all \(\ulvh \in\Uhklzr\) as
	\[
		\calE_h(w; \ulvh) \defeq -\brac[\Omega]{\nabla\cdot(\matK\nabla w), \vh} - \aKh(\Ihkl w, \ulvh),
	\]
	for any \(w\in \HONEzr( \Omega) \) such that \(\nabla \cdot(\matK\nabla w)\in \LTWO(\Omega)\). For all \(r\in\{0,\dots,k\}\) and such a \(w\) that additionally satisfies \(w\in\HS{r+2}(\Th)\), the consistency error satisfies
	\begin{equation}\label{eq:cons.error}
		\norm[\aKh, *]{\calE_h(w; \cdot)} \lesssim \errorRHS{w}. 
	\end{equation}
\end{lemma}

\begin{proof} 
	The following equality has been established in the proof of \cite[Lemma 3.15]{di-pietro.droniou:2020:hybrid}
	\begin{equation}\label{eq:cons.error.equality}
		\calE_h(w; \ulvh)  = \sum_{T\in\Th}\brac[\bdryT]{\matKT\nabla(w-\piKTe{k+1}w)\cdot\norT, \vFT - \vT} - \sKh(\Ihkl w, \ulvh), 
	\end{equation}
	The stabilisation term in \eqref{eq:cons.error.equality} is easily bounded due to consistency \eqref{eq:stab.consistency} and the use of Cauchy-Schwarz:
	\begin{equation}\label{eq:stab.consistency.bound}
		\seminorm{\sKh(\Ihkl w, \ulvh)}^2 \le \sKh(\Ihkl w,\Ihkl w) \sKh(\ulvh,\ulvh) 
		\lesssim\norm[\aKh]{\ulvh}^2 \sum_{T\in\Th}\locerrorRHS{w}. 
	\end{equation}
	We turn to the element-wise consistency term of \eqref{eq:cons.error.equality}. Invoking a Cauchy-Schwarz inequality first on the dot product, then on the integral yields
	\begin{align}
		\seminorm{\brac[\bdryT]{\matKT\nabla(w - \piKTe{k+1}w)\cdot\norT, \vFT - \vT}} 
		\eq \seminorm{\brac[\bdryT]{\matKT^\frac12\nabla(w-\piKTe{k+1}w), \matKT^\frac12\norT(\vFT - \vT)}} \nl
		\lea \norm[\bdryT]{\matKT^\frac12\nabla(w-\piKTe{k+1}w)}\norm[\bdryT]{\matKT^\frac12\norT(\vFT - \vT)} \nl
		\eq \hT^\frac{1}{2}\norm[\bdryT]{\matKT^\frac12\nabla(w-\piKTe{k+1}w)} \seminorm[1,\matK,\bdryT]{\ulvT} \nl
		\les \hT^\frac{1}{2}\norm[\bdryT]{\matKT^\frac12\nabla ( w-\piKTe{k+1}w)} \alphaT^\frac12 \aKT (\ulvT,\ulvT)^\frac{1}{2}, \label{eq:cons.error.proof.1}
	\end{align}
	where the last inequality of \eqref{eq:cons.error.proof.1} is due to the stability condition \eqref{eq:stab.coercivity}. We invoke the continuous trace inequality \eqref{eq:continuous.trace} on the term \(\norm[\bdryT]{\matKT^\frac12\nabla(w-\piKTe{k+1}w)}\) to yield
	\begin{equation}\label{eq:cons.error.proof.2}
		\hT^\frac{1}{2}\norm[\bdryT]{\matKT^\frac12\nabla( w-\piKTe{k+1}w)} \lesssim \seminorm[\matK,\HONE(T)]{w-\piKTe{k+1}w} + \hT\seminorm[\matK,H^2(T)]{w-\piKTe{k+1}w}. 
	\end{equation}
	It then follows from the approximation property of the oblique elliptic projector \eqref{eq:piKTe.approx} that
	\begin{equation}\label{eq:cons.error.proof.3}
		\hT^\frac{1}{2}\norm[\bdryT]{\matKT^\frac12\nabla( w-\piKTe{k+1}w)} \lesssim \hT^{r+1}\seminorm[\matK,\HS{r+2}(T)]{w}. 
	\end{equation}
	Thus, substituting \eqref{eq:cons.error.proof.3} into \eqref{eq:cons.error.proof.1} yields
	\begin{equation}\label{eq:cons.error.proof.4}
		\seminorm{\brac[\bdryT]{\matKT\nabla ( w-\piKTe{k+1}w)\cdot\norT,  \, \vFT - \vT}} 
		\lesssim \alphaT^\frac12\hT^{r+1}\seminorm[\matK,\HS{r+2}(T)]{w} \aKT (\ulvT,\ulvT)^\frac{1}{2}. 
	\end{equation}
	Invoking a triangle inequality on \eqref{eq:cons.error.equality} and applying the bounds \eqref{eq:stab.consistency.bound} and \eqref{eq:cons.error.proof.4} yields the required result:
	\begin{align}
		|\calE_h(w; \ulvh)| \les \sum_{T\in\Th}\hT^{r+1}\seminorm[\matK,\HS{r+2}(T)]{w}\alphaT^\frac12 \aKT(\ulvT, \ulvT)^\frac12 +\norm[\aKh]{\ulvh}\errorRHS{w} \nl
		\lea 2\norm[\aKh]{\ulvh}\errorRHS{w}. \nn\qquad\qedhere
	\end{align}
\end{proof}

We can now prove our three theorems on error estimates.

\begin{proof}[Proof of Theorem \ref{thm:energy.error} (energy errors)]~\\
	\textbf{Step 1}: Proof of \eqref{eq:discrete.energy}.
	
	It is clear by the definition \eqref{eq:energy.norm.def} of \(\norm[\aKh]{{\cdot}}\) that \(\aKh( \uluh,\ulvh) \) is coercive with respect to \(\norm[\aKh]{{\cdot}}\) with coercivity constant equal to \(1\). We also note that
	\[
		\ell_h(\ulvh) = \sum_{T\in\Th}\brac[T]{f, \vT} = \brac[\Omega]{f, \vh} = -(\nabla\cdot(\matK\nabla u), \vh).
	\]
	The conclusion then follows from \eqref{eq:cons.error} and the Third Strang Lemma \cite{di-pietro.droniou:2018:third} that gives
	\[
		\norm[\aKh]{\uluh - \Ihkl u} \le \sup_{\ulvh\in\Uhklzr, \ulvh\ne 0}\frac{|\calE_h(u; \ulvh)|}{\norm[\aKh]{\ulvh}}. 
	\]
	
	\noindent\textbf{Step 2}: Proof of \eqref{eq:continuous.energy}.
	
	We begin the proof by invoking a triangle inequality on each element \(T\in\Th\) as follows,
	\[
		\seminorm[\matK,\HONE(T)]{\pKT{k+1} \uluT - u} \le \seminorm[\matK,\HONE(T)]{\pKT{k+1} \uluT - \piKTe{k+1}u} + \seminorm[\matK,\HONE(T)]{\piKTe{k+1}u - u}. 
	\]
	By the approximation properties \eqref{eq:piKTe.approx} of the oblique elliptic projector, and recalling that \(\piKTe{k+1}u=\pKT{k+1}\ITkl u\), we infer
	\begin{align}
		\seminorm[\matK,\HONE(T)]{\pKT{k+1} \uluT - u}  \les \seminorm[\matK,\HONE(T)]{\pKT{k+1}(\uluT - \ITkl u)} + \hT^{r+1}\seminorm[\matK,\HS{r+2}(T)]{u} \nl
		\lea \aKT(\uluT - \ITkl u, \uluT - \ITkl u)^\frac12 + \hT^{r+1}\seminorm[\matK,\HS{r+2}(T)]{u}. \nn
	\end{align}
	Squaring this relation, summing over all \(T\in\Th\), applying the discrete energy error estimate \eqref{eq:discrete.energy}, and recalling that \(\alphaT\ge1\) yields the desired result.
\end{proof}

\begin{proof}[Proof of Theorem \ref{thm:jump.error} (Estimate on the jumps)]
	Take \(F\) an internal face between elements \(\Th[F]=\{T_1, T_2\}\) and consider
	\[
		\jump[F]{\pKh{k+1}\ulvh} = (\rmp_{\matK,T_1}^{k+1}\ul{v}_{T_1} - \rmp_{\matK,T_2}^{k+1}\ul{v}_{T_2})\big|_F 
		= (\rmp_{\matK,T_1}^{k+1}\ul{v}_{T_1} - \vF + \vF - \rmp_{\matK,T_2}^{k+1}\ul{v}_{T_2})\big|_F.
	\]
	Therefore
	\[
		\norm[F]{\jump[F]{\pKh{k+1}\ulvh}} \le \sum_{T\in\Th[F]}\norm[F]{\pKT{k+1}\ulvT - \vF}. 
	\]
	Since \(\vF = 0\) whenever \(F\in\Fhb\) is a boundary face, this relation is obviously true also for such faces -- for which \(\Th[F]\) reduces to one element. Square this relation, take \(T\in\Th\), sum this relation over \(F\in\FT\), multiply by \(\frac{\ulKT}{\alphaT\hT}\) and sum finally over \(T\in\Th\) to get	
	\begin{equation}\label{eq:jump.error.proof.1}
		\sum_{T\in\Th}\frac{\ulKT}{\alphaT\hT}\norm[\bdryT]{\jump[\bdryT]{\pKh{k+1}\ulvh}}^2 \le 4\sum_{T\in\Th}\frac{\ulKT}{\alphaT\hT}\norm[\bdryT]{\pKT{k+1}\ulvT-\vFT}^2.
	\end{equation}
	Consider, by a triangle inequality and the boundary norm equivalence \eqref{eq:norm.equiv.bdry}, together with the fact that \(\ulKT\le \matKT\norTF\cdot\norTF\),
	\begin{equation}\label{eq:jump.error.proof.2}
		\frac{\ulKT}{\alphaT\hT}\norm[\bdryT]{\pKT{k+1}\ulvT-\vFT}^2
		\lesssim \frac{\ulKT}{\alphaT\hT}\norm[\bdryT]{\pKT{k+1}\ulvT - \vT}^2 + \alphaT^{-1}\hT^{-1}\norm[\matK,\bdryT]{\vT-\vFT}^2. 
	\end{equation}
	The first term of \eqref{eq:jump.error.proof.2} may be bounded as follows,
	\begin{align}
		\frac{\ulKT}{\alphaT\hT}\norm[\bdryT]{\pKT{k+1}\ulvT - v_{T}}^2 
		\les \frac{\ulKT}{\alphaT\hT}( \hT^{-1}\norm[T]{\pKT{k+1}\ulvT - \vT}^2 + \hT\seminorm[\HONE(T)]{\pKT{k+1}\ulvT - \vT}^2) \nl
		\les \frac{\ulKT}{\alphaT}\seminorm[\HONE(T)]{\pKT{k+1}\ulvT - v_{T}}^2 
		\lesssim \alphaT^{-1}\seminorm[\matK,\HONE(T)]{\pKT{k+1}\ulvT - v_{T}}^2, \label{eq:jump.error.proof.3}
	\end{align}
	where the first line of \eqref{eq:jump.error.proof.3} is due to the continuous trace inequality \eqref{eq:continuous.trace}, the second follows from a Poincar\'{e}--Wirtinger inequality, and the last line is due to the norm equivalence \eqref{eq:norm.equiv.hr}. Consider integrating by parts the defining equation of the potential reconstruction \eqref{eq:pKT.def} to yield
	\begin{equation}\label{eq:jump.error.proof.4}
		\brac[T]{\matKT\nabla(\pKT{k+1}\ulvT - \vT) , \nabla w} = \brac[\bdryT]{\matKT^{\frac12}\norT(\vFT - \vT), \matKT^\frac{1}{2}\nabla w}. 
	\end{equation}
	Setting \(w=\pKT{k+1}\ulvT-\vT\in\POLY{k+1}(T)\) yields
	\begin{align}
		\seminorm[\matK,\HONE(T)]{\pKT{k+1}\ulvT-\vT}^2 \eq \brac[\bdryT]{\matKT^{\frac12}\norT(\vFT - \vT), \matKT^\frac{1}{2}\nabla(\pKT{k+1}\ulvT - \vT)} \nl
		\lea \norm[\matK,\bdryT]{\vFT - \vT}\norm[\bdryT]{\matKT^\frac12\nabla (\pKT{k+1}\ulvT - \vT)}. \label{eq:jump.error.proof.5}
	\end{align}
	By invoking the discrete trace inequality \eqref{eq:discrete.trace} on the second term of \eqref{eq:jump.error.proof.5} and simplifying by \(\seminorm[\matK,\HONE(T)]{\pKT{k+1}\ulvT-\vT}\) we can conclude that
	\begin{equation}\label{eq:jump.error.proof.6}
		\seminorm[\matK,\HONE(T)]{\pKT{k+1}\ulvT-\vT} \lesssim \hT^{-\frac{1}{2}}\norm[\matK,\bdryT]{\vFT - \vT}. 
	\end{equation}
	Thus, combining \eqref{eq:jump.error.proof.6}, \eqref{eq:jump.error.proof.3} and \eqref{eq:jump.error.proof.2}, and invoking the stability assumption \eqref{eq:stab.coercivity}, we infer that
	\begin{equation}\label{eq:jump.error.proof.7}
		\frac{\ulKT}{\alphaT\hT}\norm[\bdryT]{\pKT{k+1}\ulvT - \vFT}^2 \lesssim \alphaT^{-1}\hT^{-1}\norm[\matK,\bdryT]{\vFT - \vT}^2 
		\lesssim \alphaT^{-1}\seminorm[1,\matK,\bdryT]{\ulvT}^2 \lesssim \aKT(\ulvT,\ulvT). 
	\end{equation}
	Substituting \eqref{eq:jump.error.proof.7} into \eqref{eq:jump.error.proof.1} yields
	\[
		\sum_{T\in\Th}\frac{\ulKT}{\alphaT\hT}\norm[\bdryT]{\jump[\bdryT]{\pKh{k+1}\ulvh}}^2 \lesssim \norm[\aKh]{\ulvh}^2. 
	\]
	Setting \(\ulvh = \uluh - \Ihkl u\) and invoking the discrete energy error estimate \eqref{eq:discrete.energy} yields
	\begin{equation}\label{eq:jump.error.proof.8}
		\sum_{T\in\Th}\frac{\ulKT}{\alphaT\hT}\norm[\bdryT]{\jump[\bdryT]{\pKh{k+1}(\uluh - \Ihkl u)}}^2 \lesssim \sum_{T\in\Th}\locerrorRHS{u}.
	\end{equation}		
	Since \(u\in \HONEzr(\Omega)\), we have \(\jump[F]{u}=0\) for all face \(F\in\Fh\). Hence, reasoning as above,
	\[
		\sum_{T\in\Th}\frac{\ulKT}{\alphaT \hT}\norm[\bdryT]{\jump[\bdryT]{\piKhe{k+1}u}}^2 \le 4\sum_{T\in\Th}\frac{\ulKT}{\alphaT \hT}\norm[\bdryT]{\piKTe{k+1}u-u}^2.
	\]
	Using the continuous trace inequality followed by the approximation property \eqref{eq:piKTe.approx}, we infer
	\begin{equation}\label{eq:jump.error.proof.9}
		\sum_{T\in\Th}\frac{\ulKT}{\alphaT \hT}\norm[\bdryT]{\jump[\bdryT]{\piKhe{k+1}u}}^2 \lesssim \sum_{T\in\Th}\locerrorRHS{u}.
	\end{equation}	
	The result follows from \eqref{eq:jump.error.proof.9}, \eqref{eq:jump.error.proof.8}, the commutation property \(\pKh{k+1}\Ihkl=\piKhe{k+1}\) and a triangle inequality.
\end{proof}

\begin{proof}[Proof of Theorem \ref{thm:l2.error} (\(\LTWO\) error)]
	As this proof is given in \cite{di-pietro.droniou:2020:hybrid}, and we only require to show its validity under the presence of small faces, only a brief outline of the steps is provided here. Following the arguments in the proof of \cite[Theorem 2.32]{di-pietro.droniou:2020:hybrid} (see in particular Eqs. (2.72)--(2.74) therein), we obtain
	\begin{equation}\label{eq:l2.error.proof.1}
		\norm[\Omega]{\pKh{k+1}\uluh - u}^2	\le \norm[\Omega]{\piKhe{k+1}u-u}^2 + h^2\norm[\aKh]{\uluh-\Ihkl u}^2 + \norm[\Omega]{u_h-\pihzr{l}u}^2. 
	\end{equation}
	The first two terms of \eqref{eq:l2.error.proof.1} satisfy the required bound by \eqref{eq:piKTe.approx} and \eqref{eq:discrete.energy} respectively. By a fully discrete Aubin--Nitsche trick \cite{di-pietro.droniou:2018:third} the last term is bounded by 
	\begin{equation}\label{eq:l2.error.proof.2}
		\norm[\Omega]{u_h-\pihzr{l}u}^2 \le \norm[\aKh]{\uluh-\Ihkl u}\sup_{g\in \LTWO(\Omega) :\norm[\Omega]{g}\le 1}\norm[\aKh^*]{\calE_h(z_g;\cdot)}
		+ \sup_{g\in \LTWO(\Omega) :\norm[\Omega]{g}\le 1}\seminorm{\calE_h( u;\Ihkl z_g)}, 
	\end{equation}
	where \(z_g\) is the solution to the dual problem: \(\a(v,z_g)=\brac[\Omega]{g, v}\) for all \(v\in \HONEzr(\Omega)\).
	
	The first term of \eqref{eq:l2.error.proof.2} is bounded by the consistency error \eqref{eq:cons.error} (with \(r=0\)), the discrete energy error \eqref{eq:discrete.energy} and the bound \(|z_g|_{H^{2}(\Omega)}\lesssim \norm[\Omega]{g}\le1\) (see \cite{grisvard:1992:singularities}). For the second term of \eqref{eq:l2.error.proof.2} we turn to equation \eqref{eq:cons.error.equality} to infer
	\[
		\calE_h(u; \Ihkl z_g) = \sum_{T\in\Th}\brac[\bdryT]{\matKT\nabla(u - \piKhe{k+1}u)\cdot\norT,\piFTzr{k}z_g - \piTzr{l}z_g}
		- \sKh(\Ihkl u, \Ihkl z_g). \label{eq:l2.error.proof.3}
	\]
	The stability term in \eqref{eq:l2.error.proof.2} is bounded by the use of a Cauchy--Schwarz inequality and the stability consistency assumption \eqref{eq:stab.consistency}. It only remains to bound the first term, which is done by writing
	\begin{align}
		|\brac[\bdryT]{\matKT\nabla(u-\piKhe{k+1}u)\cdot\norT, \piFTzr{k}z_g - \piTzr{l}z_g}| 
		\les \norm[\bdryT]{\matKT\nabla(u - \piKhe{k+1}u)}\norm[\bdryT]{\piFTzr{k}z_g - \piTzr{l}z_g} \nl
		\les \hT^{r+1}\seminorm[\HS{r+2}(T)]{u}\hT\seminorm[H^2(T)]{z_g}, \nn
	\end{align}
	where the first term in the second line has been bounded by \eqref{eq:continuous.trace} and \eqref{eq:piKTe.approx}, and the second by \eqref{eq:piFT.approx}. This concludes the proof.
\end{proof}

\section{Examples of Local Stabilisation Forms}\label{sec:examples.sT}

In this section we introduce and analyse several stabilisation bilinear forms that satisfy Assumption \ref{assum:stability}. This assumption implies in particular the polynomial consistency of \(\sKT\), that is, \(\sKT(\ITkl w,\ITkl w)=0\) whenever \(w\in\POLY{k+1}(T)\) (apply \eqref{eq:stab.consistency} to \(v=w\)). Following the arguments in \cite[Lemma 2.11]{di-pietro.droniou:2020:hybrid}, this means that \(\sKT\) depends only on the difference operators \(\deltaKT{l}:\UTkl\to\POLY{l}(T)\) and \(\deltaKFT{k}:\UTkl\to\POLY{k}(\FT)\) defined by: for all \(\ulvT\in\UTkl\),
\begin{equation}\label{eq:difference.operators.def}
	\deltaKT{l} \ulvT \defeq \piTzr{l}(\pKT{k+1}\ulvT - \vT) \quad\textrm{and}\quad \deltaKFT{k} \ulvT \defeq \piFTzr{k}(\pKT{k+1}\ulvT - \vFT). 
\end{equation}
We will show that the following choices of stabilisation satisfy the design assumptions.
\begin{enumerate}
	\item \emph{Minimally scaled gradient-based stabilisation}
	\begin{align}\label{eq:stab.grad.min}
		\sKTgradmin(\uluT, &\,\ulvT) \defeq
		\ulKT\left[\brac[T]{\nabla\deltaKT{l}\uluT,\nabla\deltaKT{l}\ulvT} 
		+ \hT^{-1}\brac[\bdryT]{\deltaKFT{k}\uluT,\deltaKFT{k}\ulvT}\right].		
	\end{align}
	\item \emph{Maximally scaled gradient-based stabilisation}
	\begin{align}\label{eq:stab.grad.max}
		\sKTgradmax(\uluT, &\,\ulvT) \defeq
		\olKT\left[\brac[T]{\nabla\deltaKT{l}\uluT,\nabla\deltaKT{l}\ulvT} 
		+ \hT^{-1}\brac[\bdryT]{\deltaKFT{k}\uluT,\deltaKFT{k}\ulvT}\right].
	\end{align}
	\item \emph{Boundary stabilisation}
	\begin{equation}\label{eq:stab.bdry}
		\sKTbdry(\uluT, \ulvT) \defeq \hT^{-1}\brac[\matK, \bdryT]{\deltaKFT{k}\uluT - \deltaKT{l}\uluT, \deltaKFT{k}\ulvT - \deltaKT{l}\ulvT}. 
	\end{equation}
\end{enumerate}
The stabilisation \eqref{eq:stab.bdry} is equivalent to the ``original HHO stabilisation'' as described in \cite{di-pietro.droniou:2020:hybrid}, with a change of scaling for face differences from \(\hF^{-1}\) to \(\hT^{-1}\). This change is however critical to ensure that Assumption \ref{assum:stability} is satisfied with hidden constants that do not depend on the number or smallness of the faces in \(T\). The two gradient-based stabilisation \eqref{eq:stab.grad.min} and \eqref{eq:stab.grad.max} are identical except for changes in the diffusion scaling factors. It is clear by the inequality \(\ulKT \le \olKT\) that
\begin{equation}
	\sKTgradmin(\ulvT,\ulvT) \le \sKTgradmax(\ulvT,\ulvT). \label{eq:grad.stab.bound}
\end{equation}
Moreover, if we can prove that \(\sKTgradmin\) and \(\sKTgradmax\) satisfy Assumption \ref{assum:stability}, then it follows that any choice of volumetric and boundary scaling that are bounded below by \(\ulKT\) and above by \(\olKT\) will also result in gradient-based stabilisation forms that satisfy Assumption \ref{assum:stability}.

\begin{remark}[Scaling stabilisations]
	The factors \(\hT^{-1}\) (for boundary terms) and \(1\) (for volumetric terms) in \eqref{eq:stab.grad.min}, \eqref{eq:stab.grad.max}  and \eqref{eq:stab.bdry} may both be scaled by a positive constant, or replaced by an equivalent quantity. For highly distorted meshes it is sometimes sensible to replace the scaling \(\hT^{-1}\) with \(\frac{\seminorm[d-1]{\bdryT}}{\seminorm[d]{T}}\) (c.f. \cite{droniou:2020:interplay}). The optimal choice of scaling is not discussed here and remains a topic for further research.	
\end{remark}

\subsection{Relationship Between Stabilisation Bilinear Forms}
We first establish some relations between the bilinear forms, which will facilitate their analysis.

\begin{proposition}\label{prop:stab.relate}
	The stabilisation bilinear forms defined by \eqref{eq:stab.grad.min} and \eqref{eq:stab.bdry} satisfy the relationship
	\begin{equation}\label{eq:stab.relate}
		\sKTgradmin(\ulvT,\ulvT) \lesssim \sKTbdry(\ulvT, \ulvT) \qquad\forall\ulvT\in\UTkl. 
	\end{equation}
\end{proposition}

This proposition directly results from the following two lemmas.

\begin{lemma}
	It holds that
	\begin{equation}\label{eq:deltaKT.bound}
		\seminorm[\matK,\HONE(T)]{\deltaKT{l} \ulvT} \lesssim \hT^{-\frac{1}{2}} \norm[\matK,\bdryT]{\deltaKFT{k} \ulvT - \deltaKT{l}\ulvT}\qquad\forall \ulvT\in\UTkl. 
	\end{equation}
\end{lemma}

\begin{proof}
	We begin by integrating by parts the left-hand side of the potential reconstruction equation \eqref{eq:pKT.def} with a generic \(w\in\POLY{k+1}(T)\), and by rearranging to yield
	\begin{equation}\label{eq:deltaKT.bound.proof.1}
		\brac[T]{\vT-\pKT{k+1}\ulvT, \nabla \cdot(\matKT \nabla w)} = \brac[\bdryT]{\vFT-\pKT{k+1}\ulvT , \matKT\nabla w \cdot \norT}. 
	\end{equation}
	As \(\nabla\cdot (\matKT \nabla w) \in \POLY{k-1}(T) \subseteq \POLY{l}(T)\) and \(\matKT\nabla w\cdot \norT \in \POLY{k}(\FT)\) we may introduce projections to \eqref{eq:deltaKT.bound.proof.1} as follows,
	\begin{equation}
		\brac[T]{\piTzr{l}(\vT - \pKT{k+1}\ulvT), \nabla\cdot(\matKT\nabla w)} = \brac[\bdryT]{\piFTzr{k}(\vFT - \pKT{k+1}\ulvT), \matKT\nabla w \cdot\norT}, \nn
	\end{equation}
	which gives
	\begin{equation}\label{eq:deltaKT.bound.proof.2}
		\brac[T]{\deltaKT{l}\ulvT, \nabla\cdot(\matKT\nabla w)} = \brac[\bdryT]{\deltaKFT{k}\ulvT, \matKT\nabla w \cdot\norT}. 
	\end{equation}
	Integrating by parts again yields,
	\begin{align}
		-(\matKT \nabla \deltaKT{l} \ul{v}_{T},\nabla w )_T \eq ( \deltaKFT{k}\ulvT-\deltaKT{l}\ulvT, 
		\matKT\nabla w \cdot\norT)_{\bdryT} \nl
		\eq \brac[\bdryT]{\matKT^\frac12\norT(\deltaKFT{k}\ulvT-\deltaKT{l}\ulvT),\matKT^\frac12\nabla w}. \nn
	\end{align}
	Choosing \(w=-\deltaKT{l}\ulvT\in\POLY{k+1}(T)\) we infer that
	\begin{align}
		\seminorm[\matK,\HONE(T)]{\deltaKT{l} \ulvT}^2 \lea \norm[\matK,\bdryT]{\deltaKFT{k}\ulvT-\deltaKT{l}\ulvT} \norm[\bdryT]{\matKT^\frac12\nabla\deltaKT{l} \ulvT} \nl
		\les \hT^{-\frac{1}{2}} \norm[\matK,\bdryT]{\deltaKFT{k}\ulvT-\deltaKT{l}\ulvT} \seminorm[\matK,\HONE(T)]{\deltaKT{l} \ulvT}, \nn
	\end{align}
	where the last line follows from the discrete trace inequality \eqref{eq:discrete.trace}. Simplifying yields the desired result.	
\end{proof}
\begin{lemma}\label{lem:deltaKFT.bound}
	It holds that
	\[
		\ulKT^\frac12\norm[\bdryT]{\deltaKFT{k} \ulvT} \lesssim \norm[\matK,\bdryT]{\deltaKFT{k} \ulvT - \deltaKT{l}\ulvT}\qquad \forall\ulvT\in\UTkl. \label{eq:deltaKFT.bound}
	\]
\end{lemma}

\begin{proof}
	Applying a triangle inequality on \(\ulKT^{\frac12}\norm[\bdryT]{\deltaKFT{k} \ulvT}\) as well as the bound \(\ulKT\le \matKT\norTF\cdot\norTF\) yields
	\begin{align}
		\ulKT^\frac12\norm[\bdryT]{\deltaKFT{k} \ulvT} \lea \norm[\matK,\bdryT]{\deltaKFT{k} \ulvT - \deltaKT{l} \ulvT} + \ulKT^\frac12\norm[\bdryT]{\deltaKT{l} \ulvT} \nl
		\les \norm[\matK,\bdryT]{\deltaKFT{k} \ulvT - \deltaKT{l} \ulvT} + \ulKT^\frac12\hT^{-\frac12}\norm[T]{\deltaKT{l} \ulvT},
	\end{align}
	where the second line is due to the discrete trace inequality \eqref{eq:discrete.trace}. The proof is completed by invoking the Poincar\'{e}-Wirtinger inequality
	\[
		\ulKT^\frac12 \hT^{-\frac{1}{2}}\norm[T]{\deltaKT{l} \ulvT} \lesssim \ulKT^\frac12 \hT^{\frac{1}{2}}\seminorm[\HONE(T)]{\deltaKT{l} \ulvT} \lesssim \hT^{\frac{1}{2}}\seminorm[\matK,\HONE(T)]{\deltaKT{l} \ulvT}
	\]
	and applying \eqref{eq:deltaKT.bound}.
\end{proof}

\subsection{Stabilisation Properties}

We now prove that the gradient-based and boundary stabilisations satisfy the coercivity and consistency properties in Assumption \ref{assum:stability}.

\begin{lemma}[Coercivity of \(\sKTgradmin\)]\label{lem:stab.coercivity}
	The minimally scaled gradient-based stabilisation \(\sKTgradmin\) defined by \eqref{eq:stab.grad.min} satisfies the coercivity condition \eqref{eq:stab.coercivity}.
\end{lemma}

\begin{proof}
	We start by invoking a triangle inequality on \(\seminorm[1,\matK,\bdryT]{\ulvT}^2\) to write
	\begin{equation}\label{eq:stab.coercivity.proof.1}
	\seminorm[1,\matK,\bdryT]{\ulvT}^2 \lesssim \hT^{-1}\norm[\matK,\bdryT]{\vFT - \pKT{k+1}\ulvT}^2 + \hT^{-1}\norm[\matK,\bdryT]{\pKT{k+1}\ulvT - \vT}^2. 
	\end{equation}
	We bound the second term in the right-hand side by applying the norm equivalence \eqref{eq:norm.equiv.bdry} and the discrete trace inequality \eqref{eq:discrete.trace}, followed by a Poincar\'{e}--Wirtinger inequality (noting that \(\pKT{k+1}\ulvT-\vT\) has integral \(0\) due to \eqref{eq:pKT.closure}), to obtain
	\begin{equation}\label{eq:stab.coercivity.proof.2}
		\hT^{-1}\norm[\matK,\bdryT]{\pKT{k+1}\ulvT - \vT}^2 \lesssim \olKT\hT^{-2}\norm[T]{\pKT{k+1}\ulvT - \vT}^2 
		\lesssim \olKT\seminorm[\HONE(T)]{\pKT{k+1}\ulvT - \vT}^2. 
	\end{equation}
	We invoke another triangle inequality and apply the boundedness property \eqref{eq:piTzr.approx} of the \(\LTWO\)-projector as follows,
	\begin{align}
		\olKT\seminorm[\HONE(T)]{\pKT{k+1}\ulvT - \vT}^2 \les \olKT\seminorm[\HONE(T)]{\pKT{k+1}\ulvT - \piTzr{l}\pKT{k+1}\ulvT}^2 
		+ \olKT\seminorm[\HONE(T)]{\piTzr{l}\pKT{k+1}\ulvT - \vT}^2 \nl
		\les \alphaT\seminorm[\matK,\HONE(T)]{\pKT{k+1}\ulvT}^2 + \alphaT\ulKT\seminorm[\HONE(T)]{\deltaKT{l}\ulvT}^2, \nn
	\end{align}
	where we have used the norm equivalence \eqref{eq:norm.equiv.hr} and \(\olKT=\alphaT\ulKT\). The right-hand side is bounded above by \(\alphaT\aKT(\ulvT,\ulvT)\) (with the stabilisation choice \(\sKTgradmin\)), and it only remains to bound the first term in the right-hand side of \eqref{eq:stab.coercivity.proof.1}. By applying a triangle inequality and the norm equivalence \eqref{eq:norm.equiv.bdry}, we infer that
	\begin{align}
		\hT^{-1}\norm[\matK,\bdryT]{\vFT - \pKT{k+1}\ulvT}^2 \les \hT^{-1}\norm[\matK,\bdryT]{\vFT - \piFTzr{k}\pKT{k+1}\ulvT}^2 
		+ \hT^{-1}\norm[\matK,\bdryT]{\piFTzr{k}\pKT{k+1}\ulvT - \pKT{k+1}\ulvT}^2 \nl
		\les \alphaT\left( \frac{\ulKT}{\hT}\norm[\bdryT]{\deltaKFT{k}\ulvT}^2 + \seminorm[\matK,\HONE(T)]{\pKT{k+1}\ulvT}^2 \right), \label{eq:stab.coercivity.proof.3}
	\end{align}
	where we have used \(\olKT=\alphaT\ulKT\) and invoked \eqref{eq:piFT.weighted.approx} with \(l=k+1\), \(s=1\) and \(v=\pKT{k+1}\ulvT\). This concludes the proof, since the bracketed term in the right-hand side is bounded above by \(\aKT(\ulvT,\ulvT)\).
\end{proof}

\begin{corollary}[Coercivity of \(\sKTbdry\) and \(\sKTgradmax\)]
	The boundary stabilisation form \(\sKTbdry\) and maximally scaled gradient-based stabilisation form \(\sKTgradmax\) defined by \eqref{eq:stab.bdry} and \eqref{eq:stab.grad.max}, respectively, satisfy the coercivity condition \eqref{eq:stab.coercivity}.
\end{corollary}

\begin{proof}
	Invoke the estimates \eqref{eq:stab.relate} and \eqref{eq:grad.stab.bound}, and apply Lemma \ref{lem:stab.coercivity}.
\end{proof}

\begin{lemma}[Consistency of \(\sKTbdry\)]\label{lem:stab.bdry.consistency}
	The boundary stabilisation form \(\sKTbdry\) defined by \eqref{eq:stab.bdry} satisfies the consistency condition \eqref{eq:stab.consistency}.
\end{lemma}

\begin{proof}
	Recalling that \(\pKT{k+1}\circ\ITkl=\piKTe{k+1}\), and invoking equation \eqref{eq:piFT.weighted.approx} (with \(s=1\)), we may write for \(v\in\HS{r+2}(T)\),
	\[
		\sKTbdry(\ITkl v, \ITkl v) = \hT^{-1}\norm[\matK,\bdryT]{\piFTzr{k}(\piKTe{k+1}v-v) - \piTzr{l}(\piKTe{k+1}v-v)}^2 
		\lesssim \alphaT\seminorm[\matK,\HONE(T)]{\piKTe{k+1}v-v}^2. 
	\]
	The conclusion follows from the approximation properties \eqref{eq:piKTe.approx} of \(\piKTe{k+1}\). 
\end{proof}

\begin{lemma}[Consistency of \(\sKTgradmax\)]\label{lem:stab.grad.consistency}
	The maximally scaled gradient-based stabilisation form \(\sKTgradmax\) defined by \eqref{eq:stab.grad.max} satisfies the consistency condition \eqref{eq:stab.consistency}.
\end{lemma}

\begin{proof}
	By the boundedness \eqref{eq:piTzr.approx} of \(\piTzr{l}\) and \eqref{eq:piFT.bound} of \(\piFTzr{k}\), and a continuous trace inequality \eqref{eq:continuous.trace} we have for all \(v\in \HS{r+2}(T)\),
	\begin{align}
		\sKTgradmax(\ITkl v, \ITkl v) \eq \olKT\seminorm[\HONE(T)]{\piTzr{l}(\piKTe{k+1}v-v)}^2 + \olKT\hT^{-1}\norm[\bdryT]{\piFTzr{k}(\piKTe{k+1}v-v)}^2 \nl 
		\les \olKT\seminorm[\HONE(T)]{\piKTe{k+1}v-v}^2 + \hT^{-1}\olKT\seminorm[\bdryT]{\piKTe{k+1}v-v}^2 \nl
		\les \olKT\seminorm[\HONE(T)]{\piKTe{k+1}v-v}^2+\hT^{-2}\olKT\norm[T]{\piKTe{k+1}v-v}^2. \nn
	\end{align}
	Applying a Poincar\'{e}--Wirtinger inequality followed by the norm equivalence \eqref{eq:norm.equiv.hr} yields
	\[
		\sKTgradmax(\ITkl v, \ITkl v) \lesssim \alphaT\seminorm[\matK,\HONE(T)]{\piKTe{k+1}v-v}^2. 
	\]
	The conclusion follows from \eqref{eq:piKTe.approx}.
\end{proof}

The consistency of \(\sKTgradmin\) follows from either of the previous two lemmas and \eqref{eq:stab.relate} or \eqref{eq:grad.stab.bound}.

\begin{corollary}[Consistency of \(\sKTgradmin\)]
	The minimally scaled gradient-based stabilisation form \(\sKTgradmin\) defined via \eqref{eq:stab.grad.min} satisfies the consistency condition \eqref{eq:stab.consistency}.
\end{corollary}

\subsection{The Case $l=k-1$}

We show in this section that for the special case \(l=k-1\), and with identity diffusion tensor \(\matK=\mat{I}\), the volumetric term of the gradient-based stabilisations need not be considered. This is inline with results found in VEM \cite{beirao-da-veiga.lovadini:2017:stability, brenner.sung:2018:virtual}. We define the alternative boundary stabilisation \(\sKTkminus:\UTkl\times\UTkl\to\R\) such that, for all \(\uluT,\ulvT\in\UTkl\),
\begin{equation}\label{eq:stab.kminus1}
	\sKTkminus(\uluT,\ulvT) \defeq \hT^{-1}\brac[\bdryT]{\deltaKFT{k}\uluT,\deltaKFT{k}\ulvT}. 
\end{equation}
It is clear that \(\sKTkminus(\ulvT,\ulvT) \le \sKTgradmax(\ulvT,\ulvT)\),
which, with Lemma \ref{lem:stab.grad.consistency}, implies the consistency of \(\sKTkminus\). Therefore, we only have to prove its coercivity.

\begin{lemma}[Coercivity of \(\sKTkminus\)]
	The boundary stabilisation form \(\sKTkminus\) defined via \eqref{eq:stab.kminus1} satisfies the coercivity condition \eqref{eq:stab.coercivity}.
\end{lemma}

\begin{proof}
	Following the same procedure as in the proof of Lemma \ref{lem:stab.coercivity}, we may apply a triangle inequality to \(\seminorm[1,\matK, \bdryT]{\ulvT}\) and bound the first term of \eqref{eq:stab.coercivity.proof.1} in the same manner as \eqref{eq:stab.coercivity.proof.3}. Thus, it remains to be shown that
	\[
		\hT^{-1}\norm[\bdryT]{\pKT{k+1}\ulvT-\vT}^2 \lesssim \aKT(\ulvT,\ulvT)
	\]
	with the stabilisation choice \(\sKTkminus\). We begin with equation \eqref{eq:stab.coercivity.proof.2} with \(\matK=\mat{I}\) and a triangle inequality to yield
	\[
		\hT^{-1}\norm[\bdryT]{\pKT{k+1}\ulvT-\vT}^2  \lesssim \hT^{-2}\norm[T]{\pKT{k+1}\ulvT-\piTzr{k-1}\pKT{k+1}\ulvT}^2 + \hT^{-2}\norm[T]{\deltaKT{k-1}\ulvT}^2. 
	\]
	By invoking the approximation properties of the \(\LTWO\) projector \eqref{eq:piTzr.approx} we conclude that
	\begin{equation}\label{eq:stab.kminus1.coercivity.proof.1}
		\hT^{-1}\norm[\bdryT]{\pKT{k+1}\ulvT-\vT}^2 \lesssim \seminorm[\HONE(T)]{\pKT{k+1}\ulvT}^2 +  \hT^{-2}\norm[T]{\deltaKT{k-1}\ulvT}^2. 
	\end{equation}
	Consider equation \eqref{eq:deltaKT.bound.proof.2} with \(l=k-1\) and \(\matK=\mat{I}\),
	\begin{equation}\label{eq:stab.kminus1.coercivity.proof.2}
		\brac[T]{\deltaKT{k-1}\ulvT,  \Delta w} = \brac[\bdryT]{\deltaKFT{k}\ulvT , \nabla w \cdot \norT} \lesssim \hT^{-\frac12}\norm[\bdryT]{\deltaKFT{k}\ulvT}\norm[T]{\nabla w}, 
	\end{equation}
	where the inequality follows from a Cauchy--Schwarz inequality and a discrete trace inequality \eqref{eq:discrete.trace}. The operator \(\Delta:\POLY{k+1}(T)\to\POLY{k-1}(T) \) is onto so we may choose \(w\in\POLY{k+1}(T)\) such that \(\Delta w = \deltaKT{k-1}\ulvT\) and \(\norm[T]{\nabla w} \lesssim \hT \norm[T]{\Delta w}\) (the factor \(\hT\) follows from a simple scaling argument). Substituting into \eqref{eq:stab.kminus1.coercivity.proof.2} and rearranging yields
	\begin{equation}\label{eq:stab.kminus1.coercivity.proof.3}
		\norm[T]{\deltaKT{k-1}\ulvT}^2 \lesssim \hT\norm[\bdryT]{\deltaKFT{k}\ulvT}^2. 
	\end{equation}
	Combining equations \eqref{eq:stab.kminus1.coercivity.proof.1} and \eqref{eq:stab.kminus1.coercivity.proof.3} yields the desired result.
\end{proof}

\begin{remark}[Generic diffusion]
	It is possible to define a boundary stabilisation term equivalent to \eqref{eq:stab.kminus1} with a generic diffusion tensor \(\matK\) and appropriate scaling. However, the coercivity constant may have a greater than desired dependence on the diffusion anisotropy \(\alphaT\).
\end{remark}

\begin{remark}[The case \((k,l)=(0,0)\)]
	When \(l=0\), the element difference operator is identically \(0\), i.e. \(\deltaKT{0}\ulvT\equiv0\) for all \(\ulvT\in\UTkl\). Thus, for the case where the cell unknowns are constants, all stabilisation forms are defined entirely in terms of the boundary difference operator. Therefore, the alternative boundary stabilisation \(\sKTkminus\) is still valid for \(l=k=0\), and is indeed equal to the boundary stabilisation \(\sKTbdry\) defined by \eqref{eq:stab.bdry}.
\end{remark}

\section{Numerical Results}

We provide here a variety of numerical tests for the scheme \eqref{eq:discrete.problem} on meshes with small faces. The method is implemented using the \texttt{HArDCore} open source C++ library \cite{hhocode}. We solve the linear system using the BiCGSTAB solver found in the \texttt{Eigen} library, with documentation available at \url{https://eigen.tuxfamily.org/dox/index.html}. All numerical tests are conducted on coarse meshes generated by the agglomeration of triangular and rectangular (in 2D) and cubic (in 3D) meshes. In order to produce these meshes we have written a C++ mesh agglomeration package in the \texttt{HArDCore} library.

The accuracy of each scheme is measured by the following relative energy errors,
\[
	E_{\aKh} \defeq \frac{\norm[\aKh]{\uluh - \Ihkl u}}{\norm[\aKh]{\Ihkl u}} \quad\textrm{and}\quad E_{\matK,1} \defeq \frac{\seminorm[\matK,\HONE(\mathcal{T}_h)]{u - \pKh{k+1}u}}{\seminorm[\matK,\HONE(\mathcal{T}_h)]{u}},
\]
and measure of the jumps, 
\[
	E_{\matK,J} \defeq \left(\sum_{T\in\Th} \frac{\ulKT}{\hT} \norm[\bdryT]{\jump[\bdryT]{\pKh{k+1}\uluh}}^2\right)^{\frac12}. 
\]
We wish to show that these measures of error converge optimally as the face diameter \(\hF\) gets arbitrarily small compared to the cell diameter \(\hT\). To quantify this relative smallness we define the regularity parameter \(\gamma_h\) to be the average of the ratio \(\frac{\hT}{\hF}\),
\[
	\gamma_h\defeq\frac{1}{\CARD{\Th}}\sum_{T\in\Th}\frac{1}{\CARD{\FT}}\sum_{F\in\FT}\frac{\hT}{\hF}. 
\]

\subsection{Tests in Two Dimensions}

We conduct all two-dimensional tests in the unit square \(\Omega=(0,1)^2\) with exact solution \(u\) given by
\[
	u(x,y) = \sin(\pi x)\sin(\pi y). 
\]
The accuracy of the scheme is illustrated by testing with various meshes, stabilisations, diffusion tensors, and approximation orders.

\subsubsection{Test A.}

We consider here an identity diffusion tensor \(\matK=\mat{I}\), and a mesh sequence \(\Mh\) such that \(\gamma_h\to \infty\) as \(h\to 0\). The parameters of the mesh sequence are given in Table \ref{table:meshA} and two of the meshes are plotted in Figure \ref{fig:meshA}. As both gradient-based stabilisations defined by \eqref{eq:stab.grad.max} and \eqref{eq:stab.grad.min} are equal for identity diffusion we denote here their shared value by \(\sKTgrad\). We also consider the following stabilisation term analogous to \cite[Example 2.8]{di-pietro.droniou:2020:hybrid},
\[
	\sKT'(\uluT, \ulvT) \defeq \hT^{-2}\brac[T]{\deltaKT{l} \uluT, \deltaKT{l} \ulvT} + \hT^{-1}\brac[\bdryT]{\deltaKFT{k} \uluT, \deltaKFT{k} \ulvT}. 
\]

\begin{figure}[H]
	\centering
	\includegraphics[width=4\textwidth/9]{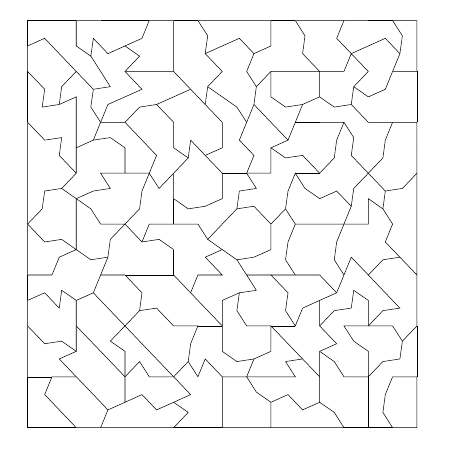} 
	\includegraphics[width=4\textwidth/9]{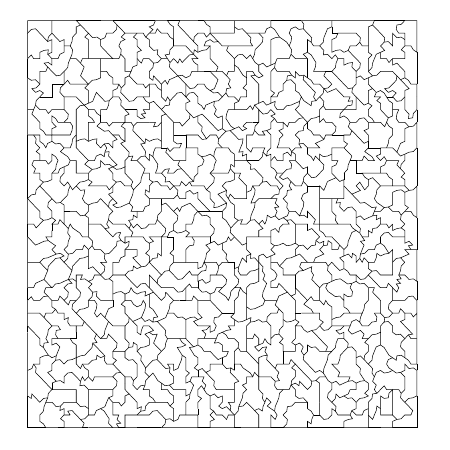} 
	\caption{Two members of the mesh sequence used in Tests A and B.}
	\label{fig:meshA}
\end{figure}

\begin{table}[H]
	\centering
	\pgfplotstableread{data/mesh1_sTori_id_k1.dat}\loadedtable
	\pgfplotstabletypeset
	[
	columns={MeshSize,hThF,NbCells,NbInternalEdges,FacesPerCell}, 
	columns/MeshSize/.style={column name=\(h\)},
	columns/hThF/.style={column name=\(\gamma_h\)},
	columns/NbCells/.style={column name=Nb. Elements},
	columns/NbInternalEdges/.style={column name=Nb. Internal Edges},
	columns/FacesPerCell/.style={column name=Avg. \(\textrm{Card}(\FT)\)},
	every head row/.style={before row=\toprule,after row=\midrule},
	every last row/.style={after row=\bottomrule} 
	]\loadedtable
	\caption{Parameters of mesh sequence used in Tests A and B.}
	\label{table:meshA}
\end{table}

The graphs of the induced errors versus the mesh diameter \(h\) are plotted on a log-log scale in Figure \ref{fig:testA} comparing the various choices of stabilisation term and polynomial degree \(k\). The asymptotic convergence rates match those predicted by the theory in Section \ref{sec:error.estimates}. In particular, these results show that the accuracy of the HHO scheme is independent, as \(h\to 0\), of the increasing number of faces or their relative smallness. Here and throughout we consider \(l=k\) except when using the stabilisation term \(\sKTkminus\) which requires the reduced polynomial degree \(l=k-1\) on element unknowns. It is worth noting that despite this reduction of order, \(\sKTkminus\) still performs well compared to the alternate choices of stabilisation term. Indeed, the varying choices of stability produce similar \(\HONE\) and jump errors, whereas the error in the energy norm is notably worse for the choice of stabilisation \(\sKT'\). In light of this we define \(E_{0,\Th}\) and \(E_{0,\Fh}\) to measure the \(\LTWO\) errors induced from the cell and face unknowns respectively,
\[
	E_{0,\Th} \defeq \left(\frac{\sum_{T\in\Th}\norm[T]{\piTzr{l}(\uT-u)}^2}{\sum_{T\in\Th}\norm[T]{\piTzr{l}u}^2}\right)^{\frac12}\ ;\quad E_{0,\Fh} \defeq \left(\frac{\sum_{F\in\Fh}\hF\norm[F]{\piFzr{k}(\uF-u)}^2}{\sum_{F\in\Fh}\hF\norm[F]{\piFzr{k}u}^2}\right)^{\frac12}.
\]
In Figure \ref{fig:testA2} we compare the errors induced by the choices of stabilisation term \(\sKT'\) and \(\sKTgrad\) with respect to the measures \(E_{0,\Th}\) and \(E_{0,\Fh}\).  The errors induced by the face unknowns are identical, whereas \(\sKTgrad\) performs far better with respect to the measure \(E_{0,\Th}\). This suggests the choice of volumetric term in \(\sKT'\) is poor.
\begin{figure}[H]
	\centering
	\ref{legend:testA}
	\vspace{0.5cm}\\
	\subcaptionbox{\(E_{\aKh}\) vs. \(h\), \(k=1\)}
	{
		\begin{tikzpicture}[scale=0.57]
		\begin{loglogaxis}[ legend columns=4, legend to name=legend:testA ]
		\addplot table[x=MeshSize,y=EnergyError] {data/mesh1_sTori_id_k1.dat};
		\addplot table[x=MeshSize,y=EnergyError] {data/mesh1_sTgrad_id_k1.dat};
		\addplot table[x=MeshSize,y=EnergyError] {data/mesh1_sTbdry_id_k1.dat};
		\addplot table[x=MeshSize,y=EnergyError] {data/mesh1_sTvol_id_k1.dat};
		\logLogSlopeTriangle{0.90}{0.4}{0.1}{2}{black};
		\legend{\(\sKTbdry\), \(\sKTgrad\), \(\sKTkminus\), \(\sKT'\)};
		\end{loglogaxis}
		\end{tikzpicture}
	}
	\subcaptionbox{\(E_{\matK,1}\) vs. \(h\), \(k=1\)}
	{
		\begin{tikzpicture}[scale=0.57]
		\begin{loglogaxis}
		\addplot table[x=MeshSize,y=H1Error] {data/mesh1_sTori_id_k1.dat};
		\addplot table[x=MeshSize,y=H1Error] {data/mesh1_sTgrad_id_k1.dat};
		\addplot table[x=MeshSize,y=H1Error] {data/mesh1_sTbdry_id_k1.dat};
		\addplot table[x=MeshSize,y=H1Error] {data/mesh1_sTvol_id_k1.dat};
		\logLogSlopeTriangle{0.90}{0.4}{0.1}{2}{black};
		\end{loglogaxis}
		\end{tikzpicture}
	}
	\subcaptionbox{\(E_{J,h}\) vs. \(h\), \(k=1\)}
	{
		\begin{tikzpicture}[scale=0.57]
		\begin{loglogaxis}
		\addplot table[x=MeshSize,y=JumpError] {data/mesh1_sTori_id_k1.dat};
		\addplot table[x=MeshSize,y=JumpError] {data/mesh1_sTgrad_id_k1.dat};
		\addplot table[x=MeshSize,y=JumpError] {data/mesh1_sTbdry_id_k1.dat};
		\addplot table[x=MeshSize,y=JumpError] {data/mesh1_sTvol_id_k1.dat};
		\logLogSlopeTriangle{0.90}{0.4}{0.1}{2}{black};
		\end{loglogaxis}
		\end{tikzpicture}
	}
	\subcaptionbox{\(E_{\aKh}\) vs. \(h\), \(k=3\)}
	{
		\begin{tikzpicture}[scale=0.57]
		\begin{loglogaxis}
		\addplot table[x=MeshSize,y=EnergyError] {data/mesh1_sTori_id_k3.dat};
		\addplot table[x=MeshSize,y=EnergyError] {data/mesh1_sTgrad_id_k3.dat};
		\addplot table[x=MeshSize,y=EnergyError] {data/mesh1_sTbdry_id_k3.dat};
		\addplot table[x=MeshSize,y=EnergyError] {data/mesh1_sTvol_id_k3.dat};
		\logLogSlopeTriangle{0.90}{0.4}{0.1}{4}{black};
		\end{loglogaxis}
		\end{tikzpicture}
	}
	\subcaptionbox{\(E_{\matK,1}\) vs. \(h\), \(k=3\)}
	{
		\begin{tikzpicture}[scale=0.57]
		\begin{loglogaxis}
		\addplot table[x=MeshSize,y=H1Error] {data/mesh1_sTori_id_k3.dat};
		\addplot table[x=MeshSize,y=H1Error] {data/mesh1_sTgrad_id_k3.dat};
		\addplot table[x=MeshSize,y=H1Error] {data/mesh1_sTbdry_id_k3.dat};
		\addplot table[x=MeshSize,y=H1Error] {data/mesh1_sTvol_id_k3.dat};
		\logLogSlopeTriangle{0.90}{0.4}{0.1}{4}{black};
		\end{loglogaxis}
		\end{tikzpicture}
	}
	\subcaptionbox{\(E_{J,h}\) vs. \(h\), \(k=3\)}
	{
		\begin{tikzpicture}[scale=0.57]
		\begin{loglogaxis}
		\addplot table[x=MeshSize,y=JumpError] {data/mesh1_sTori_id_k3.dat};
		\addplot table[x=MeshSize,y=JumpError] {data/mesh1_sTgrad_id_k3.dat};
		\addplot table[x=MeshSize,y=JumpError] {data/mesh1_sTbdry_id_k3.dat};
		\addplot table[x=MeshSize,y=JumpError] {data/mesh1_sTvol_id_k3.dat};
		\logLogSlopeTriangle{0.90}{0.4}{0.1}{4}{black};
		\end{loglogaxis}
		\end{tikzpicture}
	}
	\caption{Test A.}
	\label{fig:testA}
\end{figure}

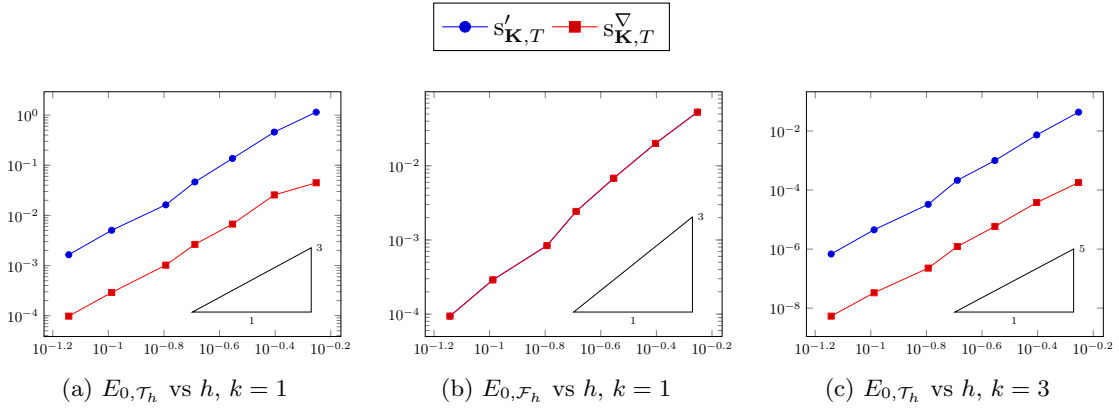
\begin{figure}[H]
	\centering
	\ref{legend:testA2}
	\vspace{0.5cm}\\
	\subcaptionbox{\(E_{0,\Th}\) vs \(h\), \(k=1\)}
	{
		\begin{tikzpicture}[scale=0.57]
		\begin{loglogaxis}[ legend columns=2, legend to name=legend:testA2 ]
		\addplot table[x=MeshSize,y=CellL2] {data/voltest_vol_ht_k1.dat};
		\addplot table[x=MeshSize,y=CellL2] {data/voltest_grad_ht_k1.dat};
		\logLogSlopeTriangle{0.90}{0.4}{0.1}{3}{black};
		\legend{\(\sKT'\), \(\sKTgrad\)};
		\end{loglogaxis}
		\end{tikzpicture}
	}
	\subcaptionbox{\(E_{0,\Fh}\) vs \(h\), \(k=1\)}
	{
		\begin{tikzpicture}[scale=0.57]
		\begin{loglogaxis}
		\addplot table[x=MeshSize,y=FaceL2] {data/voltest_vol_ht_k1.dat};
		\addplot table[x=MeshSize,y=FaceL2] {data/voltest_grad_ht_k1.dat};
		\logLogSlopeTriangle{0.90}{0.4}{0.1}{3}{black};
		\end{loglogaxis}
		\end{tikzpicture}
	}
	\subcaptionbox{\(E_{0,\Th}\) vs \(h\), \(k=3\)}
	{
		\begin{tikzpicture}[scale=0.57]
		\begin{loglogaxis}
		\addplot table[x=MeshSize,y=CellL2] {data/voltest_vol_ht_k3.dat};
		\addplot table[x=MeshSize,y=CellL2] {data/voltest_grad_ht_k3.dat};
		\logLogSlopeTriangle{0.90}{0.4}{0.1}{5}{black};
		\end{loglogaxis}
		\end{tikzpicture}
	}
	\caption{\(\LTWO\) error vs \(h\)}
	\label{fig:testA2}
\end{figure}	

\subsubsection{Test B.}
This test is designed to verify convergence of the scheme with an anisotropic diffusion tensor
\[
	\matK = \begin{pmatrix} \lambda & 0 \\ 0 & \lambda^{-1} \end{pmatrix} \nn
\]
with \(\lambda = 10^2\). The mesh sequence considered for these tests are the same as those for Test A. We compare in Figure \ref{fig:testB} the two gradient-based defined by \eqref{eq:stab.grad.max} and \eqref{eq:stab.grad.min} as well as an alternate gradient-based stabilisation defined by 
\[
	\sKT^{\nabla,\matK}(\uluT, \ulvT) \defeq \brac[T]{\matKT\nabla\deltaKT{l} \uluT, \nabla\deltaKT{l} \ulvT} + \hT^{-1}\brac[\matK, \bdryT]{\deltaKFT{k} \uluT, \deltaKFT{k} \ulvT}. \nn
\]

\begin{figure}[H]
	\centering
	\ref{legendb}
	\vspace{0.5cm}\\
	\subcaptionbox{\(E_{\aKh}\) vs. \(h\), \(k=0\)}
	{
		\begin{tikzpicture}[scale=0.56]
		\begin{loglogaxis}
		\addplot table[x=MeshSize,y=EnergyError] {data/diffusion_gradmin_k0.dat};
		\addplot table[x=MeshSize,y=EnergyError] {data/diffusion_gradmax_k0.dat};
		\addplot table[x=MeshSize,y=EnergyError] {data/diffusion_gradktf_k0.dat};
		\logLogSlopeTriangle{0.90}{0.4}{0.1}{1}{black};
		\end{loglogaxis}
		\end{tikzpicture}
	}
	\subcaptionbox{\(E_{\matK,1}\) vs. \(h\), \(k=0\)}
	{
		\begin{tikzpicture}[scale=0.56]
		\begin{loglogaxis}
		\addplot table[x=MeshSize,y=H1Error] {data/diffusion_gradmin_k0.dat};
		\addplot table[x=MeshSize,y=H1Error] {data/diffusion_gradmax_k0.dat};
		\addplot table[x=MeshSize,y=H1Error] {data/diffusion_gradktf_k0.dat};
		\logLogSlopeTriangle{0.90}{0.4}{0.1}{1}{black};
		\end{loglogaxis}
		\end{tikzpicture}
	}
	\subcaptionbox{\(E_{J,h}\) vs. \(h\), \(k=0\)}
	{
		\begin{tikzpicture}[scale=0.56]
		\begin{loglogaxis}
		\addplot table[x=MeshSize,y=JumpError] {data/diffusion_gradmin_k0.dat};
		\addplot table[x=MeshSize,y=JumpError] {data/diffusion_gradmax_k0.dat};
		\addplot table[x=MeshSize,y=JumpError] {data/diffusion_gradktf_k0.dat};
		\logLogSlopeTriangle{0.90}{0.4}{0.1}{1}{black};
		\end{loglogaxis}
		\end{tikzpicture}
	}
	\subcaptionbox{\(E_{\aKh}\) vs. \(h\), \(k=1\)}
	{
		\begin{tikzpicture}[scale=0.56]
		\begin{loglogaxis}
		\addplot table[x=MeshSize,y=EnergyError] {data/diffusion_gradmin_k1.dat};
		\addplot table[x=MeshSize,y=EnergyError] {data/diffusion_gradmax_k1.dat};
		\addplot table[x=MeshSize,y=EnergyError] {data/diffusion_gradktf_k1.dat};
		\addplot table[x=MeshSize,y=EnergyError] {data/diffusion_ori_k1.dat};
		\logLogSlopeTriangle{0.90}{0.4}{0.1}{2}{black};
		\end{loglogaxis}
		\end{tikzpicture}
	}
	\subcaptionbox{\(E_{\matK,1}\) vs. \(h\), \(k=1\)}
	{
		\begin{tikzpicture}[scale=0.56]
		\begin{loglogaxis}
		\addplot table[x=MeshSize,y=H1Error] {data/diffusion_gradmin_k1.dat};
		\addplot table[x=MeshSize,y=H1Error] {data/diffusion_gradmax_k1.dat};
		\addplot table[x=MeshSize,y=H1Error] {data/diffusion_gradktf_k1.dat};
		\addplot table[x=MeshSize,y=H1Error] {data/diffusion_ori_k1.dat};
		\logLogSlopeTriangle{0.90}{0.4}{0.1}{2}{black};
		\end{loglogaxis}
		\end{tikzpicture}
	}
	\subcaptionbox{\(E_{J,h}\) vs. \(h\), \(k=1\)}
	{
		\begin{tikzpicture}[scale=0.56]
		\begin{loglogaxis}
		\addplot table[x=MeshSize,y=JumpError] {data/diffusion_gradmin_k1.dat};
		\addplot table[x=MeshSize,y=JumpError] {data/diffusion_gradmax_k1.dat};
		\addplot table[x=MeshSize,y=JumpError] {data/diffusion_gradktf_k1.dat};
		\addplot table[x=MeshSize,y=JumpError] {data/diffusion_ori_k1.dat};
		\logLogSlopeTriangle{0.90}{0.4}{0.1}{2}{black};
		\end{loglogaxis}
		\end{tikzpicture}
	}
	\subcaptionbox{\(E_{\aKh}\) vs. \(h\), \(k=2\)}
	{
		\begin{tikzpicture}[scale=0.56]
		\begin{loglogaxis}[ legend columns=4, legend to name=legendb ]
		\addplot table[x=MeshSize,y=EnergyError] {data/diffusion_gradmin_k2.dat};
		\addlegendentry{\(\sKTgradmin\)}
		\addplot table[x=MeshSize,y=EnergyError] {data/diffusion_gradmax_k2.dat};
		\addlegendentry{\(\sKTgradmax\)}
		\addplot table[x=MeshSize,y=EnergyError] {data/diffusion_gradktf_k2.dat};
		\addlegendentry{\(\sKT^{\nabla,\matK}\)}
		\addplot table[x=MeshSize,y=EnergyError] {data/diffusion_ori_k2.dat};
		\addlegendentry{\(\sKTbdry\)}
		\logLogSlopeTriangle{0.90}{0.4}{0.1}{3}{black};
		\end{loglogaxis}
		\end{tikzpicture}
	}
	\subcaptionbox{\(E_{\matK,1}\) vs. \(h\), \(k=2\)}
	{
		\begin{tikzpicture}[scale=0.56]
		\begin{loglogaxis}
		\addplot table[x=MeshSize,y=H1Error] {data/diffusion_gradmin_k2.dat};
		\addplot table[x=MeshSize,y=H1Error] {data/diffusion_gradmax_k2.dat};
		\addplot table[x=MeshSize,y=H1Error] {data/diffusion_gradktf_k2.dat};
		\addplot table[x=MeshSize,y=H1Error] {data/diffusion_ori_k2.dat};
		\logLogSlopeTriangle{0.90}{0.4}{0.1}{3}{black};
		\end{loglogaxis}
		\end{tikzpicture}
	}
	\subcaptionbox{\(E_{J,h}\) vs. \(h\), \(k=2\)}
	{
		\begin{tikzpicture}[scale=0.56]
		\begin{loglogaxis}
		\addplot table[x=MeshSize,y=JumpError] {data/diffusion_gradmin_k2.dat};
		\addplot table[x=MeshSize,y=JumpError] {data/diffusion_gradmax_k2.dat};
		\addplot table[x=MeshSize,y=JumpError] {data/diffusion_gradktf_k2.dat};
		\addplot table[x=MeshSize,y=JumpError] {data/diffusion_ori_k2.dat};
		\logLogSlopeTriangle{0.90}{0.4}{0.1}{3}{black};
		\end{loglogaxis}
		\end{tikzpicture}
	}
	\subcaptionbox{\(E_{\aKh}\) vs. \(h\), \(k=3\)}
	{
		\begin{tikzpicture}[scale=0.56]
		\begin{loglogaxis}
		\addplot table[x=MeshSize,y=EnergyError] {data/diffusion_gradmin_k3.dat};
		\addplot table[x=MeshSize,y=EnergyError] {data/diffusion_gradmax_k3.dat};
		\addplot table[x=MeshSize,y=EnergyError] {data/diffusion_gradktf_k3.dat};
		\addplot table[x=MeshSize,y=EnergyError] {data/diffusion_ori_k3.dat};
		\logLogSlopeTriangle{0.90}{0.4}{0.1}{4}{black};
		\end{loglogaxis}
		\end{tikzpicture}
	}
	\subcaptionbox{\(E_{\matK,1}\) vs. \(h\), \(k=3\)}
	{
		\begin{tikzpicture}[scale=0.56]
		\begin{loglogaxis}
		\addplot table[x=MeshSize,y=H1Error] {data/diffusion_gradmin_k3.dat};
		\addplot table[x=MeshSize,y=H1Error] {data/diffusion_gradmax_k3.dat};
		\addplot table[x=MeshSize,y=H1Error] {data/diffusion_gradktf_k3.dat};
		\addplot table[x=MeshSize,y=H1Error] {data/diffusion_ori_k3.dat};
		\logLogSlopeTriangle{0.90}{0.4}{0.1}{4}{black};
		\end{loglogaxis}
		\end{tikzpicture}
	}
	\subcaptionbox{\(E_{J,h}\) vs. \(h\), \(k=3\)}
	{
		\begin{tikzpicture}[scale=0.56]
		\begin{loglogaxis}
		\addplot table[x=MeshSize,y=JumpError] {data/diffusion_gradmin_k3.dat};
		\addplot table[x=MeshSize,y=JumpError] {data/diffusion_gradmax_k3.dat};
		\addplot table[x=MeshSize,y=JumpError] {data/diffusion_gradktf_k3.dat};
		\addplot table[x=MeshSize,y=JumpError] {data/diffusion_ori_k3.dat};
		\logLogSlopeTriangle{0.90}{0.4}{0.1}{4}{black};
		\end{loglogaxis}
		\end{tikzpicture}
	}
	\caption{Test B.}
	\label{fig:testB}
\end{figure}
It is clear that the minimally scaled gradient-based stabilisation \(\sKTgradmin\) is a poor choice of stabilisation term, even with a relatively mild diffusion anisotropy of \(10^4\). The induced errors can be several orders of magnitude greater than those induced by other choices of stabilisation term. Moreover, the solving time when using this stabilisation term can be several hundred times greater than when stabilising the scheme with \(\sKTgradmax\), and the solver fails to invert the system matrix for the fifth mesh when \(k=1\) and for the final two meshes when \(k=3\). 

The theory suggests that the \(\HONE\) and discrete energy errors, denoted by \(E_{\aKh}\) and \(E_{\matK,1}\) respectively, should scale with the square-root of the anisotropy ratio \(\alpha^\frac12 = 10^2\). As \(\alphaT\) is constant, the jump error \(E_{J,h}\) is theorised to scale like \(\alpha=10^4\). This is not observed when comparing the convergence results for \(k=1\) and \(k=3\) in Figure \ref{fig:testB} with the results from Test A (Figure \ref{fig:testA}). When, for example, we make the stabilisation choice \(\sKTbdry\), the energy and \(\HONE\) errors are not notably different for anisotropic diffusion, and the jump error sees a slight improvement.

\subsubsection{Test C.}
We consider here a mesh sequence \(\Mh\) such that \(h\) remains mostly constant, but \(\gamma_h\to \infty\). The purpose of this test is to confirm that the multiplicative constants in the error estimates of Section \ref{sec:error.estimates} are indeed not impacted by the presence of many small faces in each element. Two members of this mesh family are illustrated in Figure \ref{fig:meshC} and the parameters of the mesh sequence are given in Table \ref{table:meshC}. 
\begin{figure}[H]
	\centering
	\includegraphics[width=4\textwidth/9]{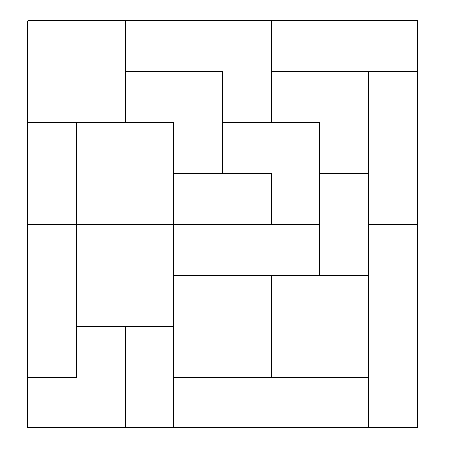} 
	\includegraphics[width=4\textwidth/9]{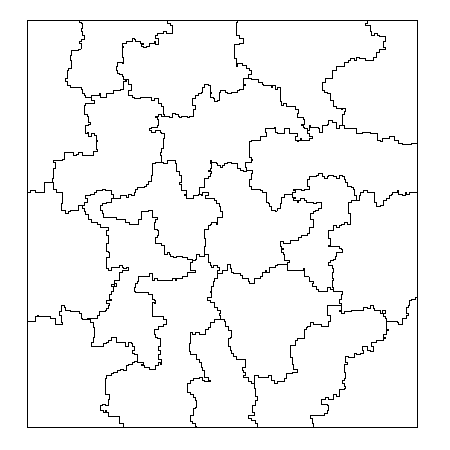} 
	\caption{Two members of the mesh sequence used in Test C.}
	\label{fig:meshC}
\end{figure}
\begin{table}[H]
	\centering
	\pgfplotstableread{data/mesh2_htk0.dat}\loadedtable
	\pgfplotstabletypeset
	[
	columns={MeshSize,hThF,NbCells,NbInternalEdges,FacesPerCell}, 
	columns/MeshSize/.style={column name=\(h\)},
	columns/hThF/.style={column name=\(\gamma_h\)},
	columns/NbCells/.style={column name=No. Elements},
	columns/NbInternalEdges/.style={column name=No. Internal Faces},
	columns/FacesPerCell/.style={column name=Avg. \(\textrm{Card}(\FT)\)},
	every head row/.style={before row=\toprule,after row=\midrule},
	every last row/.style={after row=\bottomrule}
	]\loadedtable
	\caption{Parameters of mesh sequence used in Test C.}
	\label{table:meshC}
\end{table}
We consider the stabilisation term \(\sKTbdry\) and identity diffusion \(\matK=\mat{I}\). We also consider the alternate choice of scaling in \(\sKTbdry\) replacing \(\hT\to \hF\) as is considered in \cite{di-pietro.droniou:2020:hybrid}. When scaling the stabilisation term with \(\hF\), we also make a change in the definition of \(E_{J,h}\), replacing \(\hT\to \hF\) to obtain a meaningful measurement. The results in Figure \ref{fig:testC} show that when scaling the stabilisation with \(\hT\) the error remains approximately constant as \(\gamma_h\to \infty\), as expected. The small fluctuations in error can be explained by the small changes in \(h\).
\begin{figure}[H]\centering
	\ref{legendc}
	\vspace{0.5cm}\\
	\subcaptionbox{\(E_{\aKh}\) vs. \(\gamma_h\)}
	{
		\begin{tikzpicture}[scale=0.57]
		\begin{loglogaxis}[ legend columns=2, legend to name=legendc ]
		\addplot table[x=hThF,y=EnergyError] {data/mesh2_htk0.dat};
		\addplot table[x=hThF,y=EnergyError] {data/mesh2_htk1.dat};
		\addplot table[x=hThF,y=EnergyError] {data/testC_hf_k0.dat};
		\addplot table[x=hThF,y=EnergyError] {data/testC_hf_k1.dat};
		\legend{(\(k=0\); scaling \(\hT\)),(\(k=1\); scaling \(\hT\)),(\(k=0\); scaling \(\hF\)),(\(k=1\); scaling \(\hF\))};
		\end{loglogaxis}
		\end{tikzpicture}
	}
	\subcaptionbox{\(E_{\matK,1}\) vs. \(\gamma_h\)}
	{
		\begin{tikzpicture}[scale=0.57]
		\begin{loglogaxis}[]
		\addplot table[x=hThF,y=H1Error] {data/mesh2_htk0.dat};
		\addplot table[x=hThF,y=H1Error] {data/mesh2_htk1.dat};
		\addplot table[x=hThF,y=H1Error] {data/testC_hf_k0.dat};
		\addplot table[x=hThF,y=H1Error] {data/testC_hf_k1.dat};
		\end{loglogaxis}
		\end{tikzpicture}
	}
	\subcaptionbox{\(E_{J,h}\) vs. \(\gamma_h\)}
	{
		\begin{tikzpicture}[scale=0.57]
		\begin{loglogaxis}[]
		\addplot table[x=hThF,y=JumpError] {data/mesh2_htk0.dat};
		\addplot table[x=hThF,y=JumpError] {data/mesh2_htk1.dat};
		\addplot table[x=hThF,y=JumpError] {data/testC_hf_k0.dat};
		\addplot table[x=hThF,y=JumpError] {data/testC_hf_k1.dat};
		\end{loglogaxis}
		\end{tikzpicture}
	}
	\caption{Test C.}
	\label{fig:testC}
\end{figure}
Regarding the scaling by \(\hF^{-1}\), even though it seems to display to some extent a robustness with respect to small faces, we notice that, as \(\gamma_h\) gets very large, the error appears to worsen. Moreover, the time taken to solve the system can be up to three times as long, which suggests that the system matrix may be ill conditioned by this choice of scaling. This seems to indicate that the \(\hT^{-1}\) scaling we introduced in \(\sKTbdry\) is essential to obtain an HHO scheme that is robust (both in terms of accuracy, and numerical stability) with respect to small faces.

\subsection{Tests in Three Dimensions}
In this section we conduct some 3D numerical tests in the unit box \(\Omega=(0,1)^3\). We consider the equivalent exact solution to that for the 2D tests,
\[
u(x,y,z) = \sin(\pi x)\sin(\pi y)\sin(\pi z). \nn
\]
This test is analogous to that of Test A, with identity diffusion tensor \(\matK=\mat{I}\) and \(\gamma_h\to\infty\) as \(h\to 0\). The tests are conducted for face- and element-polynomial degrees \(k=l=1\) (except when using \(\sKT^{(k-1)}\), in which case \(l=k-1=0\)). The mesh data is given in Table \ref{table:3Dmesh} and the convergence results are plotted in Figure \ref{fig:3Dtest}.
\begin{table}[H]
	\centering
	\pgfplotstableread{data/testDparams.dat}\loadedtable
	\pgfplotstabletypeset
	[
	columns={MeshSize,hThF,NbCells,NbInternalFaces, FacesPerCell}, 
	columns/MeshSize/.style={column name=\(h\)},
	columns/hThF/.style={column name=\(\gamma_h\)},
	columns/NbCells/.style={column name=No. Elements},
	columns/NbInternalFaces/.style={column name=No. Internal Faces},
	columns/FacesPerCell/.style={column name=Avg. \(\textrm{Card}(\FT)\)},
	every head row/.style={before row=\toprule,after row=\midrule},
	every last row/.style={after row=\bottomrule}
	]\loadedtable
	\caption{Parameters of mesh sequence used in Test D.}
	\label{table:3Dmesh}
\end{table}

\begin{figure}[H]
	\centering
	\ref{3dlegend2}
	\vspace{0.5cm}\\
	\subcaptionbox{\(E_{\aKh}\) vs \(h\), \(k=1\)}
	{
		\begin{tikzpicture}[scale=0.57]
		\begin{loglogaxis}[ legend columns=3, legend to name=3dlegend2 ]
		\addplot table[x=MeshSize,y=EnergyError] {data/3d_ori_ht_k1.dat};
		\addlegendentry{\(\sKTbdry\)};
		\addplot table[x=MeshSize,y=EnergyError] {data/3d_grad_ht_k1.dat};
		\addlegendentry{\(\sKTgrad\)};
		\addplot table[x=MeshSize,y=EnergyError] {data/3d_bdry_ht_k1.dat};
		\addlegendentry{\(\sKT^{(k-1)}\)};
		\logLogSlopeTriangle{0.90}{0.4}{0.1}{2}{black};
		\end{loglogaxis}
		\end{tikzpicture}
	}
	\subcaptionbox{\(E_{\matK,1}\) vs \(h\), \(k=1\)}
	{
		\begin{tikzpicture}[scale=0.57]
		\begin{loglogaxis}
		\addplot table[x=MeshSize,y=H1Error] {data/3d_ori_ht_k1.dat};
		\addplot table[x=MeshSize,y=H1Error] {data/3d_grad_ht_k1.dat};
		\addplot table[x=MeshSize,y=H1Error] {data/3d_bdry_ht_k1.dat};
		\logLogSlopeTriangle{0.90}{0.4}{0.1}{2}{black};
		\end{loglogaxis}
		\end{tikzpicture}
	}
	\subcaptionbox{\(E_{J, h}\) vs \(h\), \(k=1\)}
	{
		\begin{tikzpicture}[scale=0.57]
		\begin{loglogaxis}
		\addplot table[x=MeshSize,y=JumpError] {data/3d_ori_ht_k1.dat};
		\addplot table[x=MeshSize,y=JumpError] {data/3d_grad_ht_k1.dat};
		\addplot table[x=MeshSize,y=JumpError] {data/3d_bdry_ht_k1.dat};
		\logLogSlopeTriangle{0.90}{0.4}{0.1}{2}{black};
		\end{loglogaxis}
		\end{tikzpicture}
	}
	\caption{3D Test}
	\label{fig:3Dtest}
\end{figure}

The plots in Figure \ref{fig:3Dtest} show the error to be slightly sub-optimal for the first two meshes and otherwise converge as predicted by theory. The poor convergence rates for the first two meshes is explainable by the agglomeration process producing elements which are `less round', thus increasing the mesh regularity parameter \(\varrho^{-1}\). The results in 3 dimensions match those seen in 2 dimensions in Test A.

\begin{remark}[Increase in computation time]
	The numerical results in this section have confirmed that the error estimates are robust with respect to the size of and number of faces contained in each element. However, the number of globally coupled degrees of freedom of the system increases linearly with the total number of internal faces. Thus, the computational cost of an HHO scheme on coarse meshes is potentially far greater than on similarly scaled meshes possessing fewer faces per element. However, the process of coarsening meshes via agglomeration is still a viable option as the final computational cost is still far less than on  fine meshes.
\end{remark}

Despite testing Hybrid High-Order schemes on some highly irregular meshes, with mesh elements consisting of as many as 1000 faces, the convergence results in this section supported the theory developed in this paper. As mentioned in the introduction, generating meshes by agglomeration is a common technique for capturing complex geometries. As such, the convergence of HHO schemes on such meshes is an important addition to the literature. The numerical results in Test B, and to a lesser extent in Test C, show that the stabilisation term can have a large effect on the induced error of the scheme, and the conditioning of the system matrix. Investigating the optimal choice of stabilisation term and scaling could be an important topic for further research.

\section*{Acknowledgements}

This work was partially supported by the Australian Government through the \emph{Australian Research Council}'s Discovery Projects funding scheme (grant number DP170100605).


\bibliographystyle{plain}
\bibliography{references}


\end{document}